\documentclass{article}
\usepackage[hypertex]{hyperref}
\usepackage{diagrams}
\usepackage{amssymb}
\usepackage{graphicx}
\usepackage{amsthm}

\newcommand\pd[2]{\frac{\partial #1}{\partial #2}}
\newcommand\ddt{\frac d{dt}}
\newcommand\adx{\hbox{ad}_{\cX}}
\newcommand\Je[1]{J^{#1}(\lR^n,\lR^m)}
\newcommand\Jrr{J^\infty(\lR,\lR^2)}
\newcommand\LB{Lie-B\"{a}cklund}
\newcommand\na{|\alpha|}
\newcommand\nb{|\beta|}
\newcommand\nga{|\gamma|}
\newcommand\ds{\displaystyle}
\newcommand\ya{y^\alpha}
\newcommand\yb{y^\beta}
\newcommand\yg{y^\gamma}
\newcommand\ys{y^\sigma}
\newcommand\wsp[3]{#1^{(#2)}_{#3}}

\diagramstyle[tight,centredisplay,dpi=600,UglyObsolete,heads=LaTeX]

\newcommand\cC{{\mathcal C}}

\newcommand\cF{{\mathcal F}}
\newcommand\cG{{\mathcal G}}
\newcommand\cL{{\mathcal L}}
\newcommand\cX{{\mathcal X}}
\newcommand\cY{{\mathcal Y}}
\newcommand\cZ{{\mathcal Z}}

\newcommand\lC{{\mathbb C}}
\newcommand\lD{{\mathbb D}}
\newcommand\lR{{\mathbb R}}

\newtheorem{definition}{Definition}
\theoremstyle{plain}
\newtheorem{theorem}{Theorem}
\newtheorem{lemma}{Lemma}
\newtheorem{proposition}{Proposition}
\newtheorem{problem}{\bf Problem:}
\newtheorem*{tha}{Theorem 7(A)}
\newtheorem*{thb}{Theorem 7(B)}
\newtheorem*{thc}{Theorem 7(C)}

\title{Exponentiable \LB\ vector fields in \protect\(J^\infty(\lR^n, \lR^m)\protect\)}
\author{Ana Maria Maia Pastana \\
Sportlaan 22,
9728 PJ, Groningen - The Netherlands.\footnote{e-mail: ana.pastana@live.nl}}
\begin{document}
\maketitle
\vfill \eject
\begin{abstract}We characterize \LB\ vector fields in infinite dimensional jet bundles \(J^\infty(\lR^n,\lR^m)\) that can be exponentiated to flows with each component depending on a finite set of variables. We show that for \(m=1\) each such field is an extension of one in a finite dimensional jet space. For  \(m>1\) this is no longer the case and we give necessary and sufficient conditions for exponentiation. Non-trivial examples are provided for \((n,m)=(1,2)\).
\end{abstract}

\section{INTRODUCTION}\label{sec:intr}

\LB\ vector fields in jet bundles are ones whose flow preserves, in a technical sense to be explained below, the jet structure of the bundles.
Such vector fields have seen application to the transformation theory of differential equations \cite{olverLG}. We here address the question of such vector field in infinite dimensional bundles seeking criteria for their exponentiability to a flow in these bundles.

Vinogradov \cite{vino:smd14.661} in 1973 introduced some results concerning local isomorphisms of finite dimensional jet bundles.

In 1976, Ibragimov and Anderson \cite{ibra-ande:SMD17.437} considered \LB\ fields in jet bundles of infinite dimension
in an attempt to classify those that are extensions of \LB\ fields in bundles of finite dimension.

In \cite{vino:SMD20.985}, Vinogradov related them to symmetries of differential equations.

Independently but in the same spirit,  in 1980, Otterson and Svetlichny \cite{otte-svet:JDE36.270} sought to systematize
the study of transformations of differential equations having obtained results similar to
those of Vinogradov \cite{vino:smd14.661,vino:SMD20.985} and Ibragimov and Anderson \cite{ibra-ande:SMD17.437}, extending them in certain situations.

Many authors have used \LB\ fields in various applications to differential equations.

Fokas  \cite{foka:JMAA68.347} and Fokas and Lagerstrom  \cite{foca-lage:JMAA74.328}
used them to study the separable solutions of the Hamilton-Jacobi equations.
Fokas and Lagerstrom  \cite{foca-lage:JMAA74.342} also used them to relate classical and quantum invariants of differential equations, their relation to symmetries, and the possibility of separation of variables.

Ibragimov  \cite{ibra:CRASP293.657} used them to show the equivalence of the K.~d.~V. equation, \(v_t = v_{xxx}+ v v_x\),
and equation \(u_t=u^3u_{xxx}\).

The cited authors make little reference, in these works, to the exponentiability of the \LB\ fields.
Ibragimov and Andeson \cite{ibra-ande:SMD17.437} suggest a method of exponentiation based on the work of
Ovsjannikov  \cite{ovsj:SMD12.1497}, however this method involves resolving partial differential equations.

We consider this somewhat inconvenient and circular, in a certain sense, that the study of the transformations of differential equations depend on the resolution of equations of the same type.

In this paper we find algebraic relation on the components of an \LB\ field sufficient for it to be exponentiable to a flow in which each component depends on a finite number of variables. From a practical point of view, these relations should be finite in number.

Concretely we show in Section \ref{sec:prim} that there are in \LB\ fields in $J^\infty(\lR^n,\lR)$ beyond those already realized in finite dimensional spaces. These were exhaustively studies by Otterson and Svetlichny \cite{otte-svet:JDE36.270}.

Contrariwise, in $J^\infty(\lR^n,\lR^m),\, m > 1$, there are exponentiable \LB\ fields that are not realized in finite dimensional jet bundles. In Section \ref{sec:segu} we present a finite number of necessary and sufficient conditions for a \LB\ field \(X\) to be exponentiable in $J^\infty(\lR^n,\lR^m),\, m > 1$.  All conditions are on flows of fields in finite dimension conveniently related to  \(X\).

In Section \ref{sec:terc}, we  present a finite number of necessary and sufficient condition on the components of a particular type of \LB\ field in $\Jrr$ to be exponentiable, constructing its flow. There is interesting geometry involved which we exploit to find  non-trivial examples  of exponentiable \LB\ fields in $\Jrr$.

This work was done in partial fulfilment for the Doctorate degree in Mathematics at the Pontef\' icia Universidade Cat\'olica of Rio de Janeiro, Brazil, 1984.

\section{Preliminaries}\label{sec:prel}

In this section we present the terminology and concepts that we use in subsequent sections, the well the mathematical results needed for our development.

Let $\alpha = (\alpha_1,\alpha_2, \dots ,\alpha_n)$ be a multiindex.
Define \(\na\) as \( \sum_{i=1}^n\alpha_i\). We indicate by  \(i\) the multiindex
$(0, \dots ,0,1,0, \dots ,0)$, where \(1\)
is in the  \(i\)-th position, and by \(0\)  the multiindex $(0, \dots  ,0)$.

Each pair $(L,\alpha)$,
where $L\in  \{1,2, \dots  ,m\}$
and \(\alpha\) is a multiindex will be called a \emph{multipair}.

For a set $A$ of multipairs defines  \(|A|=\sup\{\na \,|\,(L,a) \in  A\}\).

Every set of multipairs used here will necessarily contain the elements
$(L, 0)$ for each $L$.

Consider  a map $f:U\to \lR^m$ with $U$ open in  $\lR^n$ and with
components $f_1,\dots,f_m$. For a multipair \((L,\alpha)\), we write
\(\ds \frac{\partial^\alpha f_L}{\partial x^\alpha}\) to refer to
\(\ds \frac{\partial^{\na}}{\partial x_1^{\alpha_1}\cdots \partial x_n^{\alpha_n}}f_L\).

We reserve the letters  \(f\), \(g\) \ldots to indicate local maps from
$U \subset \lR^n$ to $\lR^m$.

Given $A$, 	a set of multipairs, consider the jet bundle
\begin{diagram}
J^A(\lR^n,\lR^m)\\
\dTo_\pi \\
\lR^n \\
\end{diagram}
where the fiber over
 $x$ is the set of equivalence classes of maps
\(f\) from neighborhoods \(V\) of \(x\) in \(\lR^n\) to \(R^m\),
defined by the relation\[ f\sim g \Leftrightarrow \frac{\partial^\alpha f_L}{\partial x^\alpha}(x)=\frac{\partial^\alpha g_L}{\partial x^\alpha}(x),\,(L,\alpha)\in  A.\]

We adopt identifications displayed by the following diagram:

\begin{diagram}
J^A(\lR^n,\lR^m) &\quad \simeq & \lR^n \times \lR^A \\
\dTo_\pi & & \dTo_\pi\\
\lR^n& \rTo & \lR^n \\
\end{diagram}
where the points of \(\lR^A \) will be designated by \(\ya_L\).

\begin{definition}\label{defjtau}
 Let  $f: U \to \lR^m$ and define:
 \begin{enumerate} \label{defj}
 \item  \(\ds J^Af:x\mapsto \left(\frac{\partial^\alpha f_L}{\partial x^\alpha}(x)\right)\), for \((L,\alpha)\in  A\),\\ \label{deftau}
 \item   \(\ds \tau^Af:x\mapsto (x,J ^Af(x))\in J^A(\lR^n,\lR^m)\),
 \item \(\Gamma^A f=\tau^Af(U)\subset J^A(\lR^n,\lR^m)\).
 \end{enumerate}
\end{definition}

In particular if  \(A\) is the set of all multipairs $(L, \alpha)$ with $\na\le k$
we denote $J^A(\lR^n,\lR^m)$ by $J^k(\lR^n,\lR^m)$ and if $A$  is the set of all
multipairs, we use the notation
$J^\infty(\lR^n,\lR^m)$.

Note that \(\Gamma^Af\) is the graph of \(J^Af\)

We work with cartesian products $\lR^A$ of the real line and whenever we say
$\Phi:E\to F$ is a map, this will mean $\Phi$ has components \(\Phi_\lambda\) such that
for each $\lambda \in A$ there is  an open set $U_\lambda$  contained in  $E_\lambda$,
a finite subproduct of   $E$  with  $\pi_\lambda:E  \to  E_\lambda$  being the projection, and
a \(\cC^\infty\) map $\ds \Psi_\lambda:U_\lambda \to \lR$
such that \(\Phi_\lambda\)  factors as
$\ds \pi_\lambda\circ \Psi_\lambda: \pi_\lambda^{-1}(U_\lambda)\to U_\lambda\to\lR$.

Let $A$  be a finite set of multipairs.

Given a map $F:J^A(\lR^n,\lR^m)\to \lR$,
we say that $f$ satisfies the differential equation
 $F=0$ if
$F(\tau^A f) = 0$ or, equivalently
$\Gamma^A f \subset F^{-1} (0)$.

Given $\Phi:J^A(\lR^n,\lR^m)\to J^A(\lR^n,\lR^m)$ a local $\lC^\infty$ map,
consider $G = F \circ \Phi$
as a new differential equation. A function
  $f$ satisfies the new equation when
$\Phi(\Gamma^A f ) \subset F^{-1}(0)$.
However $\Phi(\Gamma^A f )$
may not be $\Gamma^A g$
for some $g$,
that is,
$\Phi$  does not necessarily relate the solutions of one equation to the solutions of the other.
A proper theory of transformations of differential equations should, nevertheless, transform solution along with
the equations.
Toward this end we study a special type of
local diffeomorphism, to whit: let $\cF$,
be the family of all subsets of $J^A(\lR^n,\lR^m)$
of the form $\Gamma^A f$,
We want to study the local diffeomorphisms of  $J^A(\lR^n,\lR^m)$ that leave
$\cF$ invariant, meaning that for some $g$,
$\Phi(\Gamma^A f)= \Gamma^A g$
whenever $\Phi(\Gamma^A f)$
is a graph of  some map \(\lR^n\to\lR^m\). The precise local expression of this
idea will be given further below.

Recall that every $\cC^\infty$ vector field $\cX$
on a finite dimensional manifold $M$
is integrable, that is,  $\cX$  has a maximal \emph{flow}
which is a map $\Phi: D\to M$,
where $D = \{(x,t)
\in M\times \lR, t \in   J_x \}$
is open and each  $J_x$ is a maximum open interval containing $0$ such that
$t \mapsto  \Phi(x,t)$
is an integral curve of  $\cX$
with $\Phi(x,0)=x$.
We write $\Phi_t(x)$ for $\Phi(x,t)$.

\begin{definition}\label{defo1}
A vector fields
$\ds \cZ = \sum_{\lambda \in A} \zeta_\lambda(z)\pd{}{z_\lambda}$
 in an infinite products of  real lines
 $E=\lR^\Lambda$
 is \emph{exponentiable} when there is  a map
\(\Phi\) in $E \times \lR$,
of the previously defined type, with components $(\Phi_t)_\lambda$
such that
$\ds  \left. \frac d{dt}(\Phi_t)_\lambda\right|_{t=0}=\zeta_\lambda(z)$
and $\Phi_t\circ\Phi_s=\Phi_{t+s}$.
\end{definition}

In other words, each component \((\Phi_t)_\lambda\) of the flow of \(\cZ\)
depends on a finite set of variables that can be chosen to be independent of  $t$.

\begin{definition}\label{deftd}
  The \emph{total derivative} fields \(\lD^i\) in  $J^\infty(\lR^n,\lR^m)$
are defined as
\begin{equation}\label{totder}
 \lD^i=\pd{}{x_i}+\sum_{L,\alpha}y_L^{\alpha+i}\pd{}{y_L^\alpha}.
\end{equation}
\end{definition}

It is easily shown that \([\lD^i,\lD^j] =0\) and so for a  multiindex \(\alpha\) we define \(\lD^\alpha=(\lD^1)^{\alpha_1}\cdots (\lD^n)^{\alpha_n}\).
One also has
\begin{equation}\label{dD}
\ds \pd{}{x_i}(F\circ \tau^\infty f)=(\lD^i F)\circ \tau^\infty f.
\end{equation}

\begin{definition}\label{defSx}
  Let $X$ be a topological space, $x\in X$ and
$S\subset X$. The \emph{germ \(S_x\) of \(S\) at \(x\)} is the equivalence class of \(S\),
where \(S\sim T\) whenever there is a neighborhood \(N\) of \(x\)
such that $S \cap N = T \cap  N$.

Let  \(x\in \lR^n\) and \(\cF\) a set if functions defined on open sets of \(\lR^n\). For \(f\in \cF\), the \emph{germ}
\([f]_x\) of \(f\) at \(x\) is the equivalence class of \(f\) by which \(f\sim g\) whenever there is a neighborhood \(N\) of \(x\)
such that \(f|N=g|N\).
\end{definition}

\begin{definition}\label{LB}
A \({\cC}^\infty\) vector field \(\cX\) with flow \(\Phi_t\)
defined in a open set
 \(W\) in \(J^A(\lR^n,\lR^m)\) with finite
 \(A\) is called a \emph{Lie-B\"{a}cklund } field,  if for each
 \(f\) such that \(\Gamma^A f \subset W\)
and for each \(x_0\) in the domain of  \(f\) there exists
a \(\delta > 0\), and a neighborhood \(V\) of  \(x_0\)
such that
\begin{equation}\label{eLB}
\Phi_t\left((\Gamma^A f)_{\tau^Af(x)} \right)=(\Gamma^Af_t)_{\Phi_t\tau^Af(x)}
\end{equation}
for all \(t\), \(|t| < \delta\) and all \(x\) in \(V\) and a function \(f_t\) for each \(t\).
\end{definition}

\begin{definition}\label{fg}
For a \LB\ field \(\cX\)  as in Definition \ref{LB}, let
\(G\) be the set of germs of functions \(f:R^n\to R^m\)
such that \(\Gamma^A f\subset W\).

The map \([f]_x \mapsto [f_t]_{\Psi_t(x)}\) where
\begin{equation}\label{germflo}
 \Psi_t(x) =\pi\circ\Phi_t(\tau^Af(x))
\end{equation}
will be called the \emph{flow induced on germs}.
\end{definition}

For a flow of a \LB\ field one has
\begin{equation}\label{phtpst}
  \Phi_t\circ \tau^Af(x)=\tau^Af_t(\Psi_t(x))
\end{equation}

Given a set \(A\) of multipairs defines
\[A' = A \cup \{(L, a+i), i = 1,2, \dots ,n,  (L,a) \in   A\}\]
and \(A^{[\ell]}\) the set obtained from  \(A\)
repeating the process \(\ell\) times.

Note that if  \(\ell = \infty\),
\(A^{[\infty]}\) is the set of all multipairs.

Consider a  \LB\ vector field in
\(J^A(\lR^n,\lR^m)\), with \(A\) finite or infinite:
\begin{equation}\label{lbvf}
  \cX=\sum_{i=1}^n\xi_i\pd{}{x_i} +\sum_{(L,\alpha)\in A}\eta^\alpha_L\pd{}{\ya_L}
\end{equation}
and let \(\Phi_t=(\Xi_{it}, H^\alpha_{Lt})\)
be its flow.

With \(A\) finite and given \(\ell \ge 1\) consider the diagram:

\newarrow{Fromshto} {}{dash}{}{dash}>

\begin{diagram}
  {J^{A^{[\ell]}}(\lR^n,\lR^m)}& &\rFromshto^{\Phi_t^{[\ell]}} & &{J^{A^{[\ell]}}(\lR^n,\lR^m)} \\ %
  \dTo & & & & \dTo\\
  J^A(\lR^n,\lR^m) & &\rTo^{\Phi_t}& &J^A(\lR^n,\lR^m) \\ %
  \dTo & & & & \dTo \\
  \lR^n & & & & \lR^n
\end{diagram}

For each \( P  \in J^{A^{[\ell]}} (\lR^n ,\lR^m)\) chose a map \(f\)  such that
\(\tau^{ A^{[\ell]}} f (x)=P\). We have  \(\ds \Phi_t^{[\ell]}(P)=\tau^{ A^{[\ell]}} f_t (\Xi_{it}\circ\tau^Af(x))\)
where  \(f_t\) is such that   \(\Gamma^Af_t  =  \phi_t(\Gamma^Af)\) defining  \(\Phi_t^{[\ell]}\)  called the \emph{lifting} of \(\Phi_t\).

We know that \(\Phi_t^{[\ell]}\)
is well defined, for all \(\ell > 1\),
and is  the flow of a \LB\  field 
\(\cX^{[\ell]}\) in \(J^{A^{[\ell]}}\)
which is called a lifting of \(\cX\).

We also write
\[\Phi_t^{[\ell]}=(\Xi_{it},H^\alpha_{Lt}),\,  i = 1,2, \dots, n \hbox{\ and\ } (L,\alpha) \in A^{[\ell]}\] 	
without distinguishing the notation used for its components from those used for the components of  \(\Phi_t\).
This is justified as they in fact coincide for all \(i = 1,2, \dots ,n\) and all
\((L,\alpha) \in A\).

We  now construct a lifting of \(\cX\) and of  \(\Phi_t\) to \(J^\infty (\lR^n ,\lR^m)\). These will be called \(\cX^\infty\), with components \((\xi_i,\eta^\alpha_L)\),
and \(\Phi_t^\infty\) with components \((\Xi_{it}, H^\alpha_{Lt})\)  respectfully.

By construction  \[\Phi_t^\infty(\tau^\infty f(x))=\tau^\infty f_t(\Xi_{it}\circ \tau^\infty f(x))\] where
\[\ds f_{Lt}(\Xi_t\circ \tau^Af(x))=H^0_{Lt}\circ \tau^Af(x).\]

Applying  \(\ds \pd{}{x_i}\) to both sides of
\[\pd{{}^\alpha f_{Lt}}{x^\alpha}\left(\Xi_t\circ \tau^\infty f\right)=H^\alpha_{Lt}\circ\tau^\infty f\]
 we get:
\[\sum_j\pd{}{x_j}\partial^\alpha f_{Lt}(\Xi_t\circ \tau^\infty f)\pd{}{x_i}(\Xi_{jt}\circ \tau^\infty f)=
\pd{}{x_i}(H^\alpha_{Lt}\circ\tau^\infty f) =\]
\[=\sum_j\partial^{\alpha+j}f_{Lt}(\Xi_t\circ \tau^\infty f)\,\lD^i\Xi_{jt}\circ \tau^\infty f=\lD^i H^\alpha_{Lt}\circ\tau^\infty f,\]
that is,
\[\sum_j(\lD^i\Xi_{jt}\circ \tau^\infty f)(H^{\alpha+j}_{Lt}\circ\tau^\infty f)=\lD^i H^\alpha_{Lt}\circ\tau^\infty f.\]
Thus, the coordinates of  \(\Phi^\infty_t\) are related by the following
equations \setcounter{equation}{1}
\begin{equation}\label{eqo2}
\sum_j\lD^i\Xi_{jt}\,H^{\alpha+j}_{Lt}=\lD^iH^\alpha_{Lt}
\end{equation}

for all \(i=1,2,\dots,n\) and for all \((L,\alpha)\).

As \(\Xi_{it}\) are components of a flow, the matrix \(\lD^i\Xi_{jt}\) is invertible.
To see this concretely, the map \(x_i\mapsto \Xi_{it}\circ \tau^\infty(f)\) is an invertible map from the domain of \(f\) to that of \(f_t\) and hence, using (\ref{dD}), the matrix \(\partial^j(\Xi_{it}\circ\tau^\infty(f)=(\lD^j\Xi_{it})\circ\tau^\infty(f)\) is invertible. By Borel's lemma we can choose the function \(f\) so that at any point \(x_0\),  \(\tau^\infty(f)(x_0)\) matches any give coordinates \(\ya_L\) in \(J^\infty\) hence the conclusion follows.

We now can use  (\ref{eqo2})  to calculate
\(H^\beta_{Mt}\) for any \((M, \beta)\) starting with \(\Xi_{jt}\) and \(H^0_{Mt}\) as follows:
let \(M^j_i\)
be the inverse matrix to  \(\lD^i\Xi_{jt}\), then from  (\ref{eqo2})  we have
\[H^{\alpha+j}_{Mt} =\sum_i M^j_i\lD^i H^\alpha_{Mt}.\]

Iterating this process we have \(H^\beta_{Mt}=F^\beta_M(\Xi_{jt},H^0_{Mt})\)
where \(F^\beta_M\) is a certain differential operator.

 Expanding \(\Phi_t\)  to first order and using (\ref{eqo2}) we get:
 \begin{equation}\label{eqo3}
 \eta^{\alpha+k}_L=\lD^k\eta^\alpha_L-\sum_iy^{\alpha+i}_L\lD^k\xi_i
\end{equation}
for each unitary multiindex  \(k\)
and all   \((L,\alpha)\),  where \(\xi_i\) and  \(\eta^\alpha_L\) are, respectively, the
 coefficients of \(\ds \pd{}{x_i}\) and of  \(\ds \pd{}{\ya_L}\) in the expression of  \(\cX^\infty\).
 Note that the components of the field are wholly determined by the components
 \(\xi_i\) and \(\eta^0_L\)
which we call the\emph{ principal components}.

\begin{definition}
A vector field  \(\cX\) in \(\Je\infty\)
which satisfies (\ref{eqo3}) is called a \emph{ \LB\ }field.

If \(\cX\) is also exponentiable  (Definition \ref{defo1}), it will be called
an \emph{exponentiable } \LB\ field.
\end{definition}

Given a \LB\ field \(\cX\), then at the point \(x+t\xi\circ \tau^\infty f(x)\),
the function \(J^\infty f_t\) has, components given by
\[\left(J^\infty f_t\right)_{(L,\alpha)}(x+t\xi\circ \tau^\infty f(x))=\pd{{}^\alpha f_L}{x^\alpha}(x)+t\eta^\alpha_L\circ \tau^\infty f(x)+ o(t).\]
Now the transformation
\(x \mapsto x+t\xi\circ \tau^\infty f(x)\) has, to first order, the inverse \(x \mapsto x-t\xi\circ \tau^\infty f(x)\).
Substituting this  in the expression for  \(J^\infty f_t\) we find that at  \(x\)
it has the following components:
\[\pd{{}^\alpha f_L}{x^\alpha}(x)+t\eta^\alpha_L\circ \tau^\infty f(x)-t\sum_i\xi_i\circ\tau^\infty f(x) \pd{{}^{\alpha+i}f_L}{x^{\alpha+i}}(x) +o(t).\]
Introduce now the functions
\begin{equation}\label{eps}
  \epsilon^\alpha_L=\eta^\alpha_L-\sum_i y^{\alpha+i}_L\xi_i.
\end{equation}
Using this allows us to write
\begin{equation}\label{defeps}
\left(J^\infty f_t(x)\right)_{(L,\alpha)}=\pd{{}^\alpha f_L}{x^\alpha}(x)+t\epsilon^\alpha_L\circ \tau^\infty f(x)+ o(t),
\end{equation}
exhibiting the meaning of the functions \( \epsilon^\alpha_L\) as the change of the components of the function \(J^\infty f_t\) at
a \emph{fixed} point as opposed to \( \eta^\alpha_L\) which is the change of the components at a \emph{displaced} point.

In terms of  \(\epsilon^\alpha_L\) equations  (\ref{eqo3}) become
\begin{equation}\label{epspi}
 \epsilon^{\alpha+i}_L=\lD^i\epsilon^\alpha_L
\end{equation}
and the vector  field    \(\cX\) becomes
\[\cX=\sum_{i=1}^n\xi_i\lD^i+\sum_{(L,\alpha)}\epsilon^\alpha_L\pd{}{\ya_L}\]

We also have that
\begin{equation}\label{dxc}
[\lD^i, \cX]=\sum_j(\lD^i\xi_j)\lD^j
\end{equation}

Suppose that the flow \(\Phi_t\) of the field \(\cX\), in \(E\times F\), has the
following property: there is a map \(\Phi^\flat_t:E\to E\) such that the diagram below
commutes:
\begin{diagram}
E\times F& \rTo^{\Phi_t}& E\times F \\
\dTo_\pi && \dTo_\pi\\
E& \rTo^{\Phi_t^\flat} & E
\end{diagram}
where \(\Phi^\flat_t\) is the flow of a field \(\cX ^\flat\) in \(E\).

Note that \(\Phi_t\) is a lifting of \(\Phi^\flat_t\) but
we reserve the term ``lifting" to the particular type defined earlier, so we say in
this case that  \(\Phi_t\) and \(\cX\) are respectively extensions of
\(\Phi^\flat_t\) and \(\cX^\flat\) to \(E \times F\).

Particularly, in jet bundles, suppose  \(A \subset B\).
Then \(J^B = J^A \times F\) for some \(F\) and let \(\pi_A: J^B \to  J^A\) be the projection.

  If \(\cX\) is a \LB\ field (exponentiable, in case \(|B|=\infty\))
in \(J^B\) and there exists \(\cX^\flat\) in \(J^A\) such that \(\cX\) is an extension of  \(\cX^\flat\)
then  \(\cX^\flat\) is a \LB\ field (exponentiable, in case \(|A|=\infty\)),
seeing that
\[\Phi^\flat_t\circ \tau^Af=\Phi^\flat_t \circ\pi_A \tau^Bf = \pi_A(\Phi_t\circ \tau^Bf) =\]
\[=\pi_A(\tau^B\circ f_t\circ\Psi_t^B) = \tau^Af_t\circ \Psi_t^A\]
We say that  \(\cX\) defined on \(\Je A\) and  \(\cY\) defined on \(\Je B\)
are \emph{equivalent} in case they induce the same flow on germs.

We now cite two results that will be used in the first part of this paper and which
correspond to theorem 2 and 3 of \cite{otte-svet:JDE36.270}.

\begin{theorem}\label{th:os1}
Let  \( A\) be a finite set of  multipairs. Let \(\cX\)
be an \LB\  field in \(\Je A\). There exists a partition
\(P_1, \dots, P_q\) of \(\{1,2, ... ,m\}\) and a sequence of
 \LB\ field \(\cY_1,\dots,\cY_q\) such that:
\begin{enumerate}
\item \label{th:0s1i} \(\cY_1\) is field in \(J^{A\cap \{(L,\alpha)\,|\, L\in P_1, \na \le 1\}}\)
and for each \(j > 1\), \(\cY_j\) and is an extension of an appropriate lifting of order
\(\ell_{j-1}\) of \(\cY_{j-1}\), that is,  \(\cY_j\) is a field in
\[J^{A_j}(\lR^n,\lR^{Q_j})=J^{A^{[\ell_{j-1}]}_{j-1}}(\lR^n, \lR^{Q_{j-1}}) \times \lR^{A\cap \{(L,\alpha)\,|\,L\in P_j, \na \le 1\}}\]
where \(Q_j = P_1\cup \dots\cup P_j\) and \(A_j\) is recursively defined by this formula.
\item \label{th:os1ii}\(\cY_q\) and \(\cX\) are equivalent.
\item\label{th:os1iii} Each one of the functions \(\epsilon^0_L\) determines the components   \(\xi_i\) of
\(\cX\) by \linebreak \(\ds \xi_i=-\pd{\epsilon^0_L}{y^i_L}\).
\vskip12pt
Consider \(\{1, 2, \dots,m\}\) the a set of vertices. Whenever
\(\ds \pd{\epsilon^0_L}{{\cY}^\alpha_K}\neq 0\) for \(K\neq L\)
construct an arrow labeled by \(\alpha\) going from \(K\) to  \(L\):
\(K \stackrel{\alpha}\rightarrow L\).
\item\label{th:os1iv}In each cycle \(L_0 \stackrel{\alpha_1}{\rightarrow} L_1 \stackrel{\alpha_2}{\rightarrow} L_2 \stackrel{\alpha_3}{\rightarrow} \cdots \stackrel{\alpha_n}{\rightarrow} L_n=L_0\)
contained in the graph constructed above, one necessarily has \(\alpha_i = 0\), for all \(i\).
\end{enumerate}
\end{theorem}

Reciprocally we have:

\begin{theorem}\label{th:os2}
Given a set of functions \(\Phi_L\), \(L = 1,2,\dots ,m\) in \(\Je \infty\)
such that
\begin{enumerate}
  \item\label{th:os2i}\(\ds \pd{\Phi_L}{y^i_L}\) do not depend on \(L\).
  \item\label{th:os2ii} \(\ds \pd{\Phi_L}{\ya_L}=0\) if \(\na>1\).
  \item\label{th:os2iii}If  \(L_0 \stackrel{\alpha_1}{\rightarrow} L_1 \stackrel{\alpha_2}{\rightarrow} L_2 \stackrel{\alpha_3}{\rightarrow} \cdots \stackrel{\alpha_n}{\rightarrow} L_n=L_0\) is a cycle as described above, then, \(\alpha_i = 0\).
\end{enumerate}
Then, defining
\[\ds \xi_i =-\pd{\Phi_L}{y^i_L},\quad \epsilon^\alpha_L=\lD^\alpha \Phi_L,\quad \eta^\alpha_L= \epsilon^\alpha_L +\sum_iy^{\alpha+i}_L\xi_i\]
there is a finite set  \(A\) of multipairs such that
\[\cX=\sum_{i=1}^n\xi_i\pd{}{x_i} +\sum_{(L,\alpha)\in A}\eta^\alpha_L\pd{}{\ya_L}\]
defines a \LB\ field in \(\Je A\).
\end{theorem}

Note that for \(m = 1\) only condition
(\ref{th:os2ii}) is restrictive, thus any function
\(\Phi(x_1,\dots,x_n,y^0,y^1,\dots,y^m)\)
defines a \LB\ field in \(J^\infty(\lR^n,\lR)\).

\vfill\eject
\section{\LB\ field in \protect\(J^\infty(\lR^n,\lR)\protect\)}\label{sec:prim} 
\vskip 20pt

This part deals with some technical lemmas, being that Lemma  \ref{lem1}  is a consequence
of Lemma 4 of \cite{otte-svet:JDE36.270}. In these lemmas, \(R\) will indicate any
expression  for which there is no interest in making it explicit.

Let \(\cX\) be a vector field with flow \(\Phi_t=(\Xi_{it}, H^\alpha_{Lt})\)
such that each of its components \(\Xi_{it}\), \(H^\alpha_{Lt}\)  depend on a finite
number of variables.

Given that
\[\cX^k(g)=\left.\frac1{k!}\frac{d^k}{dt^k}\left(g\circ \Phi_t\right)\right|_{t=0},\]
\(g\circ \Phi_t\)
depends on a finite number of variables,  and derivatives in \(t\)
cannot introduce
new variables. A necessary condition for
\(\cX\) to have the above property is given by the following definition:

\begin{definition}\label{condF}
A vector field \(\cX\) in \(J^\infty\) \emph{satisfies
condition \(F\)} when for each coordinate function
 \(g = x_1,\dots \ya_L\)  etc.,
there is a finite set \(F_g\) of variables
such that for all \(k\), \(\cX^k(g)\) does not depend on variables that are not in \(F_g\). We denote by \(|F|_g\) the
maximum of \(|\alpha|\) such that \(y^\alpha_L\in F_g\).
\end{definition}
Note that \(F\) depends on the coordinate function \(g\) and is not uniquely specified.

We show below that a \LB\ field in \(J^\infty(\lR^n,\lR)\)
satisfying condition \(F\) is a lifting of  a \LB\ field in
\(J^1(\lR^n,\lR)\). We begin with a few lemmas.

\begin{lemma}\label{lem1}
If \(\Phi\) is  such that for  \(\nga>k\) one has \(\ds \pd{\Phi}{\yg}=0\),
then
\[\lD^\alpha \Phi=\sum_{\nga=k}y^{\gamma+\alpha}\pd{\Phi}{\yg}+R,\]
where \(R\) depends only on \(\yb\) if \(\nb<k+ \na\).
\end{lemma}

\begin{proof}
We use induction on  \(\na\). For \(\na =1\) we have:
\[\lD^i\Phi = \pd{\Phi}{x_i}+\sum_{\nga\le k}y^{\gamma+i}\pd{\Phi}{\yg}=
\sum_{\nga= k}y^{\gamma+i}\pd{\Phi}{\yg}+R\]
where \(R\) depends only on \(\yb\) if \(\nb<k+ 1\).

Assume the assertion is true for \(\na\),
\[\lD^{\alpha+i}\Phi = \pd{}{x_i}(\lD^\alpha\Phi)+\sum_\beta y^{\beta+i}\pd{}{\yb}
\left(\sum_{\nga= k}y^{\gamma+\alpha}\pd{\Phi}{\yg}+R\right)=\]
\[=\sum_{\nga= k}y^{\gamma+\alpha+i}\pd{\Phi}{\yg}+R\]
where \(R\) depends only on \(\yb\)  if  \(\nb<k+ \na+1\).
\end{proof}

\begin{lemma}\label{lem2}
\ \newline

 Let
 \begin{eqnarray}
  \cX &=& \sum_i\xi_i\lD^i+\sum_\beta \lD^\beta\epsilon\pd{}{\yb}\\
  a &=& \max\left\{\na \,\left|\,\pd\epsilon{\ya} \neq 0\right.\right\} \\
   b &=& \max\left\{\na \,\left|\,\exists i  \pd{\xi_i}{\ya} \neq 0, 1\le i\le n\right.\right\}
 \end{eqnarray}

Let  \(\ds \Phi = \sum_{\nga=k} \yg\Psi_\gamma+\rho\) where  \(\Psi_\gamma\) and  \(\rho\)  only  depend on
\(\yb\)  if  \(\nb<k\).  Then, if  \(k > b\),  \(\cX^\ell\Phi\)  has the following form:

i)  if  \(a > 1\),
\[  \cX ^\ell\Phi  = \sum_{\nga=k}\sum_{|\beta_1|=a }\cdots \sum_{|\beta_\ell|=a }
y^{\gamma+\beta_1+\cdots+\beta_\ell}\pd{\epsilon}{y^{\beta_1}}\cdots \pd{\epsilon}{y^{\beta_\ell}}\Psi_\gamma +R\]
where \(R\) depends only on \(\yb\) if \(\nb<\ell a+k\).

ii) if \(a \le 1\), \[  \cX ^\ell\Phi  = \sum_{\nga=k}\sum_{j_1}\cdots  \sum_{j_\ell }
y^{\gamma+j_1+\cdots+j_\ell}\left(\xi_{j_1}+\pd{\epsilon}{y^{j_1}}\right)\cdots \left(\xi_{j_\ell}+\pd{\epsilon}{y^{j_\ell}}\right) \Psi_\gamma +R\]
where \(R\) depends on \(\yb\) only if \(\nb<\ell+k\).

\end{lemma}

\begin{proof}
  We prove by induction on  \(\ell\).
  \vskip 2ex

i) \(a  >  1\):

\vskip 2ex
\noindent If  \(\ell =  1\) we have
\begin{equation}\label{lem2xphi}
  \cX \Phi =\sum_{i=1}^n\xi_i\pd\Phi{x_i} +\sum_{i=1}^n \sum_{\nga\le k}\xi_iy^{\gamma+i}\pd\Phi{\yg} +
\sum_{\nga\le k}\lD^\gamma \epsilon\pd\Phi{\yg},
\end{equation}
where the first term depends only on     \(\yb\) if \(\nb\le k\),
the second depends only on \(\yb\)  if \(\nb\le k+1\), and the
third depends only on \(\yb\) such that \(\nb\le k+a\).

Considering these observations and Lemma \ref{lem1}  applied to \(\partial^\gamma\epsilon\), we get
\[ \cX \Phi = \sum_{|\gamma|=k}\sum_{\nb=a}y^{\gamma+\beta} \pd\epsilon{\yb}\Psi_\gamma+R\]
where \(R\) depends only on
\(\ya\) if \(\na<k+a\).

Suppose the statement is true for \(\ell\),
\[\cX ^{\ell + 1}\Phi= \sum_i\xi_i\lD^i(\cX ^\ell\Phi)+ \sum_\beta \lD^\beta \epsilon \pd{}{\yb}(\cX ^\ell\Phi).\]
By the induction hypothesis and  Lemma \ref{lem1} applied to \(\partial^\beta\epsilon\), we have
\[  \cX ^{\ell+1}\Phi  = \sum_{\nga=k}\sum_{|\beta_1|=a }\cdots \sum_{|\beta_{\ell+1}|=a }
y^{\gamma+\beta_1+\cdots+\beta_{\ell+1}}\pd{\epsilon}{y^{\beta_1}}\cdots \pd{\epsilon}{y^{\beta_{\ell+1}}}\Psi_\gamma +R\]
where  \(R\) depends only on
\(\ya\) if \(\na<k+(\ell+1)a\).
\vskip 2ex

ii) \(a \leq 1\):

\vskip 2ex
\noindent If \(\ell = 1\), consider equation \ref{lem2xphi} for  \(\cX\Phi\) that appears in
item i). In this case, however, the first term depends only
on \(\yb\) if \(\nb\le k\), but the second and third terms depend on
the \(\yb\) such that \(\nb\le k+1\). From this and Lemma \ref{lem1}  applied to \(\partial^\gamma\epsilon\), we get that
\[\cX \Phi = \sum_{\nga=k}\sum_i\left(\xi_i+\pd\epsilon{y^i}\right)y^{\gamma+i}\Psi_\gamma+R\]
where \(R\) depends only on \(\yb\) if \(\nb\le k+1\).

Suppose that the statement is true for \(\ell\), then again by induction and Lemma \ref{lem1}  applied to \(\partial^\beta\epsilon\)
we have
\[\cX ^{\ell+1}\Phi = \sum_i\xi_i\pd{(\cX ^\ell\Phi)}{x_i} +\sum_i\sum_\beta \xi_i y^{\beta+i}\pd{(\cX ^\ell\Phi)}{x_i}
+ \sum_\beta \lD^\beta\epsilon \pd{(\cX ^\ell\Phi)}{\yb}=\]
\[= \sum_{\nga=k}\sum_{j_1}\cdots  \sum_{j_{\ell+1} }
y^{\gamma+j_1+\cdots+j_{\ell+1}}\left(\xi_{j_1}+\pd{\epsilon}{y^{j_1}}\right)\cdots \left(\xi_{j_{\ell+1}}+\pd{\epsilon}{y^{j_{\ell+1}}}\right) \Psi_\gamma +R\]
where \(R\) depends only on \(\yb\) if \(\nb\le k+\ell+1\).
\end{proof}

\begin{theorem}\label{th1}
Let \(\cX\) be a \LB\ field in \(J^\infty(\lR^n, \lR)\) that satisfies  condition \(F\). Then \(\cX\) is a lifting of a \LB\ field
in \(J^1(\lR^n, \lR)\)
and is therefore exponentiable.
\end{theorem}

\begin{proof}

\noindent\textbf{ 1)} Suppose  there is an \(i\) such that   \(\ds \pd{\xi_i}{\ya}\neq 0\) for some \(\alpha\) and defines \(a\) and \(b\)
as in Lemma  \ref{lem2}.

 Take \(i\) such that \(\ds\pd{\xi_i}{\ya}\neq 0\) for some \(\alpha\), \(\na=b\).

 Consider
 \[\cX ^2 x_i=\cX \xi_i  =\sum_j\xi_j\pd{\xi_i}{x_j}+\sum_j\sum_{\na\le b}\xi_jy^{\alpha+j}\pd{\xi_i}{\ya}+\sum_{\na\le b}\lD^\alpha \epsilon\pd{\xi_i}{\ya}.\]

\noindent\textbf{ 1.1)} Assume \( a > 1\). In this case, taking into account the dependence of each term on \(\yb\)  we have that
\[\cX ^2 x_i=\sum_{\na =a}\sum_{\nb =b}y^{\alpha+\beta}\pd\epsilon{\ya}\pd{\xi_i}{\yb}+R\]
where \(R\) does not depend on \(\yg\) if \(\nga \ge a+b\). The first term of this equation is a term of the third sum of the previous equation.

Thus, we can write \(\ds \cX ^2 x_i=\sum_{\nga=a+b}y^{\gamma}\Psi_\gamma +R\)
where
\begin{equation}\label{psigamma}
  \Psi_\gamma = \sum_{\begin{array}{c} \alpha',\beta'\\ \alpha'+\beta'=\gamma \\ |\alpha'|=a\\ |\beta'|=b \end{array}}\pd\epsilon{y^{\alpha'}} \pd{\xi_i}{y^{\beta'}}
  \end{equation}
where \(R\) does not depend on \(\ya\) if \(\na \ge a+b\).

This  shows that \(\Phi=\cX ^2 x_i\) satisfies the conditions of Lemma \ref{lem2}, therefore:
\begin{equation}\label{xlphi}
\cX ^\ell\Phi  = \sum_{|\gamma|=a+b}\sum_{|\beta_1|=a }\cdots \sum_{|\beta_\ell|=a }
y^{\gamma+\beta_1+\cdots+\beta_\ell}\pd{\epsilon}{y^{\beta_1}}\cdots \pd{\epsilon}{y^{\beta_\ell}}\Psi_\gamma +R,
\end{equation}
where \(R\) does not depend on \(y^\beta\) if \(|\beta|>\ell a+a+b\).

Consider \(\ell\) such that \(\ell a+a+b >\max_i |F|_{x_i}\).

Then, the coefficient of each \(\ys\) for each \(\sigma\)
such that \(|\sigma| = \ell a+a+b\), in the expression of \(\cX^\ell\Phi\),
should be null, since \(\cX\) satisfies condition F.

We show now that the supposition \(a > 1\) contradicts the statement above.

For this, let's examine the cited coefficients.

Let
\begin{eqnarray}
A_1 &=& \left\{\alpha \,|\, \na=a \hbox{\ and\ }\pd\epsilon{\ya}\neq 0\right\} \\
A_2 &=& \left\{\beta \,|\, \nb=b \hbox{\ and\ }\exists i\,\pd{\xi_i}{\yb}\neq 0\right\}.
\end{eqnarray}

Consider

\begin{eqnarray}\nonumber
\bar\alpha_1 &=& \max  \{\alpha_1  \,|\,  (\alpha_1,\alpha_2,\dots,\alpha_n)\in  A_1\}\\
\bar\alpha_2 &=&\max  \{\alpha_2  \,|\,  (\bar\alpha_1,\alpha_2,\dots,\alpha_n)\in  A_1\} \\
  \vdots &\vdots&\vdots \\
\bar\alpha_n  &=&  \max  \{\alpha_n  \,|\,  (\bar\alpha_1,\bar\alpha_2,\dots,\bar\alpha_{n-1},\alpha_n)\in  A_1\}.
\end{eqnarray}
Let,  \(\bar\alpha=(\bar\alpha_1, \bar\alpha_2 ,\dots\bar\alpha_n)\), that is, \(\bar\alpha\) is the  largest  element
 of \(A_1\)  in  anti-lexicographic order.

Proceed analogously to find \(\bar\beta\) in \(A_2\)

For \(\bar \gamma = \bar\alpha + \bar\beta\), consider expression (\ref{psigamma}) for \(\Psi_{\bar \gamma}\).

Considering the characterization of \(\bar\alpha\) and \(\bar\beta\) in terms of
anti-lexicographic order,  it's easy to see that in the sum  there is only
one non-zero term  which is \(\ds \pd{\epsilon}{y^{\bar\alpha}}\pd{\xi_i}{y^{\bar\beta}}\).

Let \(\bar\sigma=\bar\alpha +\bar\alpha +\cdots\bar\alpha +\bar\beta\)
where there are \(\ell+1\) repetitions of \(\bar\alpha\).
Any other set of multiindices \(\{\delta_1,\delta_2,\dots,\delta_{\ell+1},\beta\}\) such that \(\delta_1+\delta_2+\cdots+\delta_{\ell+1}+\beta=\bar\sigma\) must contain a \(\delta_i \not\in A_1\) or a \(\beta\not\in A_2\).

Thus,  the  coefficient of \(y^{\bar\alpha}\) in (\ref{xlphi}) is \(\ds \left(\pd{\epsilon}{y^{\bar\alpha}}\right)^{\ell+1}\pd{\xi_i}{y^{\bar\beta}}\)
which must vanish and so either \(\ds \pd{\epsilon}{y^{\bar\alpha}}=0 \hbox{\ or\ }
 \pd{\xi_i}{y^{\bar\beta}}=0\)
which is a contradiction.

 Therefore  \(a \le 1\).

\noindent\textbf{1.2)}  Assuming  \(a \le 1\) and taking into account the dependence of each  term with respect to \(\yb\),  we have  \[\cX ^2x_i=\sum_{\na=b}\sum_j\left(\xi_j+\pd\epsilon{y^j}\right)\pd{\xi_i}{\ya}y^{\alpha+j}+R\]
 where \(R\) depends only on \(\yb\)
if \(\nb<b+1\)
Thus, we can write \(\ds \cX ^2x_i=\sum_{\nga=b+1}\yg\Psi_\gamma+R\)
with \[\ds \Psi_\gamma=\sum_{\begin{array}{c} \alpha,j \\ \alpha+j=\gamma\\ \na=b\end{array}} \left(\xi_j+\pd\epsilon{y^j}\right)\pd{\xi_i}{\ya}\]
and which does not depend on \(\yb\) if \(\nb\ge b+1\)

Then, by Lemma  \ref{lem2}, with \(\Phi=\cX ^2x_i\)
 \[  \cX ^{\ell}\Phi  = \sum_{\nga=b+1}\sum_{j_1}\cdots  \sum_{j_\ell }
y^{\gamma+j_1+\cdots+j_\ell}\left(\xi_{j_1}+\pd{\epsilon}{y^{j_1}}\right)\cdots \left(\xi_{j_\ell}+\pd{\epsilon}{y^{j_\ell}}\right) \Psi_\gamma +R,\]
where \(R\) depends on \(\yb\) only if \(\nb<\ell+b+1\).

Let \(\ell\) be such that \(\ell+b+1>\max_i|F|_{x_i}\).

Then, the coefficient of each \(\ys\), for each \(\sigma\) such that \(|\sigma|=\ell+1+b\),  in the expression for \(\cX^\ell\Phi\) should be null since \(\cX\) satisfies condition F.

Let \(\bar\beta \in A_2\) be the as in the previous item.

Let now \(\bar\sigma = \bar\beta + j + \cdots + j\) for any fixed unitary multiindex
\(j\). By characterization of \(\bar\beta\) in terms of
anti-lexicographic order, there is no other unitary multiindex
\(\delta_1,\delta_2,\dots,\delta_{\ell+1}\)
nor \(\beta\) a multiindex in \(A_2\) such
that \(\beta +\delta_1+\delta_2+\cdots+\delta_{\ell+1}=\bar\sigma\)  Thus,  the coefficient of
\(y^{\bar\sigma}\) is \(\ds \left(\xi_j+\pd\epsilon{y^j}\right)^{\ell+1}\pd{\xi_i}{y^{\bar \beta}}\) and
for this to vanish either \(\ds \pd{\xi_i}{y^{\bar\beta}}=0\)
which is a contradiction, or \(\ds \xi_j=-\pd\epsilon{y^j}\).

Therefore, by Theorem \ref{th:os2}, \(\cX\) is a lifting of a \LB\ field in a jet bundle  \(J^A(\lR^n,\lR)\) for a finite \(A\),
 and Theorem \ref{th:os1} guarantees that this bundle can be chosen as being
\(J^1(\lR^n,\lR)\).

\noindent\textbf{ 2)}  Suppose that \(\ds \pd{\xi_i}{\ya}=0\) for all  \(i  = 1 , \dots ,n\) and   all
multiindices  \(\alpha\). Consider

\[\cX ^2y^0 =\sum_i\sum_j\xi_jy^i\pd{\xi_i}{x_j}+\sum_i\sum_j\xi_i\xi_jy^{j+i}+ \sum_i\xi_i\lD^i\epsilon +\sum_i\xi_i\pd\epsilon{x_i}
+ \]
\[+\sum_i\sum_\alpha \xi_iy^{\alpha+i}\pd\epsilon{\ya}+\sum_\alpha\lD^\alpha\epsilon\pd\epsilon{\ya}.\]

\noindent\textbf{ 2.1)} Suppose \(a > 1\).

Considering the dependence of each term on \(\yb\)
and using  Lemma \ref{lem1}, we have  \[\cX ^2y^0=\sum_{\na=a}\sum_{\nga=a}y^{\gamma+\alpha}\pd\epsilon{\yg}\pd\epsilon{\ya}+R,\]
where \(R\) does not depend on \(\yb\) if \(\nb \ge 2a\).

Thus,  \(\Phi  =  \cX ^2y^0\)  satisfies the  hypotheses   of  Lemma  \ref{lem2} and we have  \[  \cX ^{\ell+2}y^0  = \sum_{\nga=a}\sum_{|\beta_1|=a }\cdots \sum_{|\beta_{\ell+1}|=a }
y^{\gamma+\beta_1+\cdots+\beta_{\ell+1}}\pd\epsilon{\yg}\pd{\epsilon}{y^{\beta_1}}\cdots \pd{\epsilon}{y^{\beta_{\ell+1}}} +R.\]

Proceeding analogously to item 1.1, we have a contradiction,
therefore \(a \le 1\).

\noindent\textbf{ 2.2)} Let \(a \le 1\).

In this case: \(\ds \Phi=\cX ^2y^0 =\sum_\alpha\sum_i\left(\xi_i+\pd\epsilon{y^i}\right)\pd\epsilon{\ya}y^{\alpha+1}+R\)
where \(R\) does not depend on \(\yb\) if \(\nb\ge2\).

By Lemma \ref{lem2},
 \[\cX ^{\ell+2}y^0  = \sum_\alpha\sum_{j_1}\cdots  \sum_{j_{\ell+1} }
y^{\gamma+j_1+\cdots+j_{\ell+1}}\left(\xi_{j_1}+\pd{\epsilon}{y^{j_1}}\right)\cdots \left(\xi_{j_{\ell+1}}+ \pd{\epsilon}{y^{j_{\ell+1}}}\right)\pd\epsilon{\ya} +R.\]

Proceeding analogously to item 1.2, we have \(\ds \xi_j=-\pd\epsilon{y^j}\)
for all \(j = 1,2, \dots ,n\).

Therefore, by Theorem \ref{th:os2}, \(\cX\) is a lifting of a \LB\ field  in a finite dimensional jet bundle  and Theorem \ref{th:os1} tells us
that this can be chosen as being \(J^1(\lR^n,\lR)\).
\end{proof}

\vfill\eject

\section{Exponentiability}\label{sec:segu}
\vskip 20pt
We begin the second part of this work with an example
(observation 2, of \cite{otte-svet:JDE36.270}) that shows that not all \LB\ fields exponentiable in \(\Je\infty\), \(m > 1\) are liftings of a \LB\ field in a finite
dimensional jet bundle \(\Je A\).

Let \(n = 1\) and \(m =2\). The functions \(\Phi_1=y^{(2)}_2\) and \(\Phi_2=0\) satisfy  the  hypotheses  of  Theorem  \ref{th:os2},  therefore,  there is  a  finite set
of multipairs \(A\) such that \[\cX =\xi\pd{}x+\sum_{(L,k)\in A}\eta^{(k)}_L\pd{}{y^{(k)}_L}\] defines  a \LB\ field  in \(J^A(\lR,\lR^2)\), where \(\xi\) and \(\eta^{(k)}_L\) are given by Theorem  \ref{th:os2}.

The flow induced in germs is \((f_1,f_2)\mapsto (f_1+t f_2'', f_2)\).

If we introduce a change of variables given by \(\ds g_1 = \frac12 (f_1+f_2)\), \linebreak
\(\ds g_2 = \frac12 (f_1-f_2)\) the flow of germs for \((g_1,g_2)\) will be, therefore,
\[(g_1,g_2)\mapsto \left(g_1+\frac t2(g_1-g_2)'',g_2+\frac t2(g_1-g_2)''\right)\]
which is induced by an exponentiable \LB\ field,
\[\tilde \cX =\sum_{k=0}^\infty \frac{y_1^{k+2}-y_2^{k+2}}2\left(\pd{}{y_1^{(k)}}-\pd{}{y_2^{(k)}}\right)\]
in \(\Jrr\), whose flow is \[\left(x,y_1^{(k)},y_2^{(k)}\right) \mapsto \left(x,y_1^{(k)}+\frac t2\left(y_1^{k+2}-y_2^{k+2}\right),
y_2^{(k)}+\frac t2\left(y_1^{k+2}-y_2^{k+2}\right)\right).\]

For this field, \(\tilde\epsilon^0_1 =y_1^{(2)}-y_2^{(2)}=\tilde\epsilon^0_2\),
functions that do not satisfy the conclusion (iv) of Theorem \ref{th:os1}.
\(\tilde \cX \) is not, therefore, a lifting of a \LB\ field in   the bundle
\(J^A(\lR,\lR^2)\) with \(A\) finite.

As  there is no reason to prefer the pair \((f_1,f_2)\) to the pair
\((g_1,g_2)\), we conclude that it is necessary to situate the problem
of transformations of differentiable equations in \(\Je\infty\).

Backlund himself, in 1876, already had reached the conclusion
that to study transformations of differentiable equations,
it's necessary to consider infinite number of derivatives.

For \(\Je\infty\), however, there are difficulties  in defining
flows of fields.

Thus, consider a \LB\ field in \(\Je\infty\):\linebreak \(\ds   \cX =\sum_{(M,\beta)}\lD^\beta \epsilon^0_M\pd{}{\yb_M}\),
assuming, for convenience, that \(\xi =0\).
Consider the system of differentiable equations
\[\pd{}t (u_{Lt})=\epsilon^0_L\left(x,\pd{{}^\alpha u_L}{x^\alpha}\right)=\epsilon^0_L\circ \tau^\infty(u)\]
with \(u(0,x)=u_0(x)\) given.

If this system has a solution \(u_t(x)\), \(x\in V\subset \lR^n\), then \[\pd{}{t}\left(\pd{{}^\beta u_L}{x^\beta}\right)=\lD^\beta \epsilon^0_L\circ \tau^\infty(u).\]

Thus being, for \(x_0\)  fixed in \(V\), \(t\mapsto \tau^\infty u_t(x_0)\)
provides an integral curve of \(\cX\), passing through \(  \tau^\infty u_0(x_0)\),
in \(\Je \infty\).

Nevertheless we have not gotten the flow of \(\cX\) in \(\Je \infty\),
since the trajectory \(\tau^\infty u_t(x_0)\)
in general does not depend only on the point \(\tau^\infty u_0(x_0)\),
but possibly on the whole given initial condition \(u_0(x)\) since the trajectory is a just a curve on the graph of \(u_t(x)\) which is  determined by all of \(u_0(x)\).
We do not have uniqueness of the integral curves passing 
through \(\tau^\infty u_0(x_0)\).

Thus, even if a \LB\ field  in \(\Je\infty\)
can induce local actions in certain subsets of \(\Je\infty\);
\(\tau^\infty u_0\mapsto \tau^\infty u_t\),
we shall only study
those that are exponentiable in the previously defined sense of Definitions \ref{defo1}.

Another fact that motivates the inclusion of this hypothesis
is that finding  the integral curves presupposes solving the system of differentiable equations above and it becomes
somewhat circular  that the study of transformations of differentiable equations should depend on  solutions of  systems of the same type.

We pass now to some technical lemmas that will be used in the proof of the results in this section.

\begin{lemma}\label{lemo3}
  Let  \(\cX\)  be  a  \(\cC^\infty\) vector field  on a   manifold  \(M\)
of dimension  \(k\) and \(\Phi_t\)   the  flow of \(\cX\).  Let  \(f:M  \to  \lR\)
 be a  \(\cC^\infty\) function.
Then  \(\cX f  =  0\) if and only if \(f\circ \Phi_t=f\).
\end{lemma}

\begin{proof}
  As \(\ds \frac d{dt}(f\circ \Phi_t)=(\cX f)\circ \Phi_t\) and \(\Phi_0=Id\),
we have the result.
\end{proof}

\begin{lemma}\label{lemo4}
Let  \(\cX\)  be a \(\cC^\infty\) vector  field  in a   manifold  \(M\)
of dimension \(k\). Then the system \(\ds \pd{}{t}f_t=\cX f_t\)  has the unique solution
\(f_t(x)=f_0\circ \Phi_t(x)\), with \(f_0\)  given.
\end{lemma}

\begin{proof}
Introduce \(\ds \cX ^\sharp = \pd{}{s}-\cX \)
a vector field in \(M^\sharp = \lR \times M\) of dimension \(k+1\)
and defines in \(M^\sharp\) the function \(f^\sharp(s,x) =f_s(x)\)
for each family \(f_s\) in \(M\) and vice-versa. If \(\Phi_t(x)\) is the 
flow of \(\cX\),
then \(\Phi^\sharp_t(s,x) = (s+t, \Phi_{-t}(x))\)
is the flow of \(\cX^\sharp\).

We now show that
\(f^\sharp (s,x)\) satisfies  \(\cX^\sharp f^\sharp =0 \Leftrightarrow f_s(x)\) satisfies \(\ds \cX f_s(x)-\pd{}{s}f_s(x)=0\).

  Suppose that \(f^\sharp\) satisfies \(\cX^\sharp f^\sharp =0\).  By  Lemma \ref{lemo3},
\(f^\sharp \circ \Phi^\sharp_t = f^\sharp \Leftrightarrow \linebreak f^\sharp(s+t, \Phi_{-t}(x))=f^\sharp (s,x)\Leftrightarrow
f^\sharp(s+t, x)=f^\sharp (s,\Phi_t(x))\Rightarrow f_t(x)=f_0\circ\Phi_t(x)\).
\end{proof}

\begin{lemma}\label{lemo5}
Let  \(\cX\)  be a   \(\cC^\infty\)   vector  field  in a  manifold  \(M\) of dimension  \(m\).  Let  \(W  =  (a_i^j)\), \(1 \le i, j  \le  n\),  be  a matrix field in \(M\).

 Then  the  system
\begin{eqnarray}
\nonumber
  \pd{f_{it}}{t} &=& \cX  f_{it}+\sum_j a_i^j f_{jt} \\ \nonumber
  f_{i0} &=&0, \, 1\le i\le n
\end{eqnarray}
has a unique solution \(f_{it} = 0\), \(1 \le i \le n\).
\end{lemma}

\begin{proof}
Consider \(\{x_1,x_2,\dots,x_m\}\) variables of \(f_{it}\), \(1\le i\le n\).
	Introduce new variables, namely: \(y^1 ,y^2 , \dots y^n\) and defines
\(f_t^\sharp =\sum_j f_{jt}y^j\) where \(f_{jt}\) is a  solution   of  the system  above.
\[\pd{}{t}f^\sharp_t =\sum_j\pd{f_{jt}}{t}y^j =\sum_j\cX f_{jt}y^j+\sum_{k,j} y^ja^k_jf_{kt}=\cX f^\sharp+\sum_{k,j}y^ja^k_j\pd{f^\sharp}{y^k}.\]

  Consider   \[\cX^\sharp=   \cX   +  \sum_{j,k}y^ja^k_j\pd{}{y^k}\]   with  flow   \(\Phi^\sharp_t\).

   A   solution of \(\ds \cX ^\sharp f^\sharp_t   =   \pd{}{t}f^\sharp_t\) is \(f^\sharp_t =   f^\sharp_0\circ\Phi^\sharp_t\)  (by  Lemma  \ref{lemo4})
and  from   \(f^\sharp_0=0\)    we have    \(f_{it}  =   0\).
\end{proof}

\begin{lemma}\label{lemo6}
Let \(U \subset \lR^k\) be open and let \(L\) be the module over \(\cC^\infty(U)\) 
 of fields in \(U\) generated by \(\ds \pd{}{x_i},\,1\le i\le \ell\), that is,
\(\ds {\cY} \in L\Leftrightarrow {\cY}=\sum_{i=1}^\ell \xi_i(x)\pd{}{x_i}\).
Let \(\cX\) be a vector field in \(U\).
\(L\) is invariant  under   \({\rm ad}_\cX (\cdot)= [\cX , \cdot ]\) if and only  if, for each
\(p\), \(\ell+1\le p\le k\), \(\cX(x_p)\)
does not depend on \(x_1,\dots,x_\ell\).
\end{lemma}

\begin{proof}
Let   \begin{eqnarray*}
 \cX &=& \sum_{i=1}^\ell \xi_i\pd{}{x_i}+\sum_{j=\ell+1}^k \eta_j\pd{}{x_j}\\
 \cX (x_j)&=& \eta_j.
 \end{eqnarray*}

 Now,  \[\ds \left[\cX ,  \pd{}{x_r}\right]  =-  \left(\sum_{i=1}^\ell \pd{\xi_i}{x_r}\pd{}{x_i}+\sum_{j=\ell+1}^k \pd{\eta_j}{x_r}\pd{}{x_j}\right)\]
which,  by  hypothesis,  is in \(L\).

Thus,  \(\ds \pd{\eta_j}{x_r}=0\), \(\ell+1  \le  j  \le k\) and \(1 \le  r  \le  \ell\).
\end{proof}

\begin{lemma}\label{lemo7}
Given \(\cX\) and \(\cZ\) two vector fields, then,  for all \(n \ge 1\), \(\cZ\cX ^n\) can be written as one expression \(\ds \sum_{i=1}^n c_i{\cZ}_i\)
where \(c_i\) are expressions that involve sums and products of \(\cX\) and \(\cZ\) and \(\cZ_i ={\rm ad}_\cX ^{i-1}({\cZ})=[\cX ,[\cX , \cdots ,[\cX ,{\cZ}]]]\) and \(\cZ_1\)
identifies  \(\cZ\) itself.
\end{lemma}

\begin{proof}
\(\cZ \cX  = \cX  {\cZ} - [\cX ,{\cZ}] = \cX {\cZ}_1 - {\cZ}_2\).
Let, by induction, \(\ds {\cZ}\cX ^n=\sum_{i=1}^n c_i{\cZ}_i\)
then
\[{\cZ}\cX ^{n+1}=\left(\sum_{i=1}^n c_i{\cZ}_i\right)\cX =\sum_{i=1}^nc_i(\cX {\cZ}_i-{\cZ}_{i+1}),\]
which has the desired form.
\end{proof}

Given a \LB\ field \(\cX\) in \(\Je\infty\),
\[\cX =\sum_i\xi_i\pd{}{x_i}+\sum_{(L,\alpha)}\eta^\alpha_L\pd{}{y^\alpha_L},\]
consider
\begin{equation}\label{bigb}
  B=\left\{(M,\beta), \,|\, \exists i: \pd{\xi_i}{y_M^\beta} \neq 0 \hbox{\ or\ } \exists L: \pd{\eta^0_L}{y^\beta_M}\neq 0\right\}.
\end{equation}

As  \(\cX\) is a \LB\ field,  one has
\[\eta^\alpha_L=\lD^\alpha\eta^0_L-\lD^\alpha\left(\sum_{i=1}^ny^i_L\xi_i\right)-\sum_i y^{\alpha+i}_L\xi_i,\]
and therefore \(\ds \pd{\eta^\alpha_L}{y^\gamma_M}=0\) if \(\nga > \na+|B|\).

For each \(k \ge |B|\), \(\eta^\alpha_L\) is a polynomial in
\(\yg_M\), \(\nga >k\),
with coefficients being functions only of
\(x\) and \(\yg_M\), \(\nga \le k\). We then write \(\eta^\alpha_L=p^{k\alpha}_L +\eta^{k\alpha}_L\),
where \(\eta^{k\alpha}_L\) represents the homogeneous term
of grade zero of this
polynomial and \(p^{k\alpha}_L\)
represents the sum of the homogeneous terms of grade superior to zero.

Then we defines:

\begin{definition}\label{def5}
  \begin{equation}\label{eq5}
  \cX ^{(k)} =\sum_i\xi_i\pd{}{x_i}+\sum_L\eta^0_L\pd{}{y^0_L}+\sum_{\begin{array}{c} (L,\alpha) \\ 0\le \na \le k\end{array}} \eta^{k\alpha}_L \pd{}{\ya_L},
\end{equation}
\end{definition}
which is a vector field in \(\Je k\).

Besides this, consider the set of all the monomials that appear in  \(p^{k\alpha}_L\), for all \((L, \alpha)\)
with \(\na \le k\)
and enumerate them:
\(1,2, \dots,\nu(k)\).
Let now \(c^{k\alpha}_{\ell L}\)
be the coefficient of the \(\ell\)-th
monomial in \(p^{k\alpha}_L\). Each
 \(c^{k\alpha}_{\ell L}\) is a function of the
variables \(x_i\) and \(\yg_M\), \(\nga \le k\).
Then
\begin{definition}\label{def6}
  \begin{equation}\label{eq6}
  {\cY}^{(k)}_\ell = \sum_{(L,\alpha)}c^{k\alpha}_{\ell L}\pd{}{\ya_L},
  \end{equation}
\end{definition}
is a field in \(\Je k\)
for each  \(\ell= l,\dots ,\nu(k)\).

Section \ref{sec:terc} gives
 examples  of   fields  \(\cX^{(k)}\) and \(\cY^{(k)}_\ell\).

\begin{theorem}\label{th2}
Let \(\cX\) be a \LB\ field exponentiable in \(\Je\infty\)
and \(\Phi_t\) its flow. Given a finite set \(C\) of multipairs, there exists a real number \(k_0\) such that, for all
\(k \ge k_0\), the flow of \(\cX^{(k)}\) in \(\Je k\)
satisfies
\(\Phi^{(k)\alpha}_{Lt}=H^\alpha_{Lt}, \forall (L,\alpha)\in C\); \(\Xi_{it} =\Phi^{(k)}_{it}\),
\(i=1,2,\dots,n\).
\end{theorem}

\begin{proof}
Consider
\[A=\left\{(M,\beta) \,|\, \exists t \pd{\Xi_{it}}{\yb_M}\neq 0 \hbox{\ or\ } \pd{H^\alpha_{Lt}}{\yb_M}\neq 0, (L,\alpha)\in C\right\}.\]
Note that \(A \supset  C\),  since  \(H^\alpha_{L0}=\ya_L\)

Take \(k_0= \max \{\nga, (N,\gamma) \in A \cup B\}\), B defined previously in (\ref{bigb}).

	Besides this, \(\cY^{(k)}_\ell(\Xi_{it})=0\); \(\cY^{(k)}_\ell(H^\alpha_{Lt})=0\),
\(\forall \ell=1,2,\dots,\nu(k)\); \(\forall (L,\alpha)\in C\); \(\forall i=1,2,\dots,n\).

From
\begin{eqnarray*}
  \frac d{dt}H^\alpha_{Lt} &=& \cX  H^\alpha_{Lt} \\
 \frac d{dt}\Xi_{it} &=&  \cX  \Xi_{it}
\end{eqnarray*}
we have:
if \((L,\alpha) \in C\), and for all \(i = 1, \dots ,n\), the left-hand side of each equality only contains variables \(\yb_M\) if \((M,\beta) \in A\),
therefore the right-hand side
can not depend on other variables and, therefore:
\[\left\{\begin{array}{l} \cX  H^\alpha_{Lt}=\cX ^{(k)} H^\alpha_{Lt}, (L,\alpha)\in C \hbox{\ for all\ }k\ge k_0; \\
{\cY}^{(k)}H^\alpha_{Lt}=0\hbox{\ for all\ }k\ge k_0. \end{array}\right.\]

\[\left\{\begin{array}{l} \cX  \Xi_{it}=\cX ^{(k)}\Xi_{it} \\
{\cY}^{(k)}\Xi_{it}=0 \end{array}\right. \]

From what was shown above, we have:
\[\left\{\begin{array}{l}\ds\frac d{dt} \Xi_{it}=\cX ^{(k)} \Xi_{it}\\ \\
\ds\frac d{dt} H_{Lt}^\alpha = \cX ^{(k)}H_{Lt}^\alpha\\ \\
\Xi_{i0} = x_i\\ \\
 H_{L0}^\alpha=\ya_L
\end{array}\right.\]

And, by Lemmas \ref{lemo3} and \ref{lemo4}, we conclude that:
\[\Xi_{it}=x_i\circ \Phi^{(k)}_t=\Phi^{(k)}_{it} \hbox{\ and\ }  H_{Lt}^\alpha=\ya_L\circ \Phi^{(k)}_t = \Phi^{(k)\alpha}_{Lt}\]
\end{proof}

In other words, any finite set of components of the flow
\(\Phi_t\) in a jet bundle of infinite dimension, \(\Je\infty\),
coincides with a corresponding set of components of
\(\Phi^{(k)}\) in a finite dimension space.

Note that \(\Phi^{(k)}\) is  not necessarily a flow of a \LB\ field,
and thus \(\Phi\) is not a lifting of \(\Phi^{(k)}\).

An immediate consequence of Theorem \ref{th2} is:

\begin{theorem}\label{th3}(Uniqueness of flow).\
\ \newline
Let \(\cX\) be a \LB\ field exponentiable in \(\Je\infty\). Suppose that \(\Phi_t\) and \(\Psi_t\) are flows of \(\cX\).
Then \[\Phi_t = \Psi_t\].
\end{theorem}

\begin{proof}
 Suppose that \(\Phi_{it}\) and \(\Phi^0_{Lt}\)
only depend on variables \(x_i\) and \(\yg_M\)
with \(\nga\le a\) and that \(\Psi_{it}\) and \(\Psi^0_{Lt}\)  only depend on variables
\(x_i\) and \(\yg_M\) with \(\nga\le b\).

If \(\cX\) is exponentiable, given the set of multipairs
\(C=\{(L,0),L=1,2, \dots,m\}\), there exists \(k_0\) such that for all
\(k \ge k_0\) the flow of \(\wsp{\cX}k{}\), \(\wsp \Phi{k}t\) has  components 
 \(\wsp \Phi{k}{it}\) and \(\Phi^{(k)0}_{Lt}\)
equal to the respective components of the flow of \(\cX\) (Theorem \ref{th2}).

Take \(k = \max \{k_0, a, b\}\). Then \(\wsp \Phi k{it}=\Phi_{it}=\Psi_{it}\) \(j=1,2\), \(L=1,2,\dots, m\)
and \(\Phi^{(k)0}_{Lt}=\Phi_{Lt}=\Psi_{Lt}^0\) \(j=1,2\), \(L=1,2,\dots, m\).
Where we have the equalities \(\Phi_{it}=\Psi_{it}\), \(\Phi^0_{Lt}=\Psi^0_{Lt}\).

Then  the rest or the components are obtained from these
by the differential operator \(F^\beta_M\) and we have that \(\Phi_t\) and \(\Psi_t\)
coincide.
\end{proof}

We now show one reciprocal of Theorem \ref{th2}.

\begin{proposition}
 Let \(\cX\)  	be a \LB\  field  in
\(\Je\infty\) and \(A\) a finite set of multipairs that contains \(B\)
(defined by (\ref{bigb})).

For each \(k \ge |A|\), let \(\wsp{\cX}k{}\) and \(\wsp {\cY}k{}\)
be fields in \(\Je k\), according to equalities (\ref{eq5}) and (\ref{eq6}).

We suppose that for all \(k > |A|\) one has
\begin{eqnarray*}
  \wsp {\cY}k\ell \wsp \Phi k{it}  &\stackrel{(i)}=&0,\, \forall i, \forall \ell \\
  \wsp {\cY}k\ell  \Phi^{(k)\alpha}_{Lt}   &\stackrel{(ii)}=& 0,\,\forall (L,\alpha), \na\le k, \forall \ell .
\end{eqnarray*}

Then \(\cX\) is exponentiable in \(\Je\infty\).
\end{proposition}

\begin{proof}
Given \(C\), a finite set of multipairs, we claim there exists an integer \(k(C)\)
such that for all \(k \ge k(C)\) e have

\begin{eqnarray*}
  \wsp \Phi k{it} &=& \wsp \Phi{k(C)}{it} \\
 \Phi^{(k)\alpha}_{Lt} &=& \Phi^{(k(C))\alpha}_{Lt}, \,\forall (L,\alpha)\in C.
\end{eqnarray*}

Indeed, for each   \((L,\alpha) \in C\),  let \( k(\alpha) =   |A|  +  \na\) , then 
\(\forall k\ge k(\alpha)\)
\begin{eqnarray*}
  \hbox{(i)} &\Rightarrow& \cX  \wsp \Phi k{it}=\wsp{\cX}k{}\wsp\Phi k{it} \\
\hbox{(ii)} &\Rightarrow&  \cX  \wsp \Phi k{Lt}=\wsp{\cX}k{}\wsp\Phi k{Lt}
\end{eqnarray*}

For each \(k\ge k(\alpha)\) define \(\omega^\alpha_{Lt}(k)=\Phi^{(k)\alpha}_{Lt}-\Phi^{(k(\alpha))\alpha}_{Lt}\).
Therefore
\[\left\{\begin{array}{l}\ds \frac d{dt}\omega^\alpha_{Lt}(k)=\wsp{\cX}k{}\omega^\alpha_{Lt}(k)\\ \\
\omega^\alpha_{L0}(k)=0\end{array}\right. \Rightarrow \omega^\alpha_{Lt}(k)=0\Rightarrow \Phi^{(k)\alpha}_{Lt}=\Phi^{(k(\alpha))\alpha}_{Lt}.\]
Thus, set \(k(C)=|A| + |C|\) and the claim, is proven.

 We proceed analogously with \(\wsp \Phi k{it}\)

We now define \(\Xi_{it}=\Phi^{|A|}_{it}\) and \(H^\alpha_{Lt} =\Phi^{(k(\alpha))\alpha}_{Lt}\).

 We have  therefore:  \(\Phi_t=(\Xi_{it},H^\alpha_{Lt})\) is a  local  action of \(\lR\),
since  \(H^\alpha_{Lt}\circ\Phi_s =H^\alpha_{Lt}(\Xi_{is},H^\beta_{Ms})=H^{k\alpha}_{Lt}(\Xi^k_{is},H^{k\beta}_{Ms})=H^{k\alpha}_{L(t+s)}= H^\alpha_{L(t+s)}\)
and, analogously,\linebreak \(\Xi_{it}\circ \Phi_s =\Xi_{i(t+s)}\).

Our second claim is that the \(\Phi_t\), defined above, satisfies  \(\ds \sum_j\lD^i\Xi_{jt}H^{\alpha+j}_{Lt} =\lD^iH^\alpha_{Lt}\)
for all \(t\).

Indeed, for each \((L, \alpha)\) and each \(i\), let
\[\Delta^{\alpha i}_{Lt}=\sum_j\lD^i\Xi_{jt}H^{\alpha+j}_{Lt}-\lD^iH^\alpha_{Lt}.\]

We have \[\ds \frac d{dt}\Delta^{\alpha i}_{Lt}=\frac d{dt}\left(\sum_j\lD^i\Xi_{jt}H^{\alpha+j}_{Lt}-\lD^iH^\alpha_{Lt}\right)=\]
\[=\sum_j\left( \lD^i(\cX  \Xi_{jt})H^{\alpha+j}_{Lt}+ \lD^i \Xi_{jt}\cX  H^{\alpha+j}_{Lt}\right)-\lD^i \cX H^\alpha_{Lt}=\]
\[=\sum_j\left(\cX (\lD^i\Xi_{jt})H^{\alpha+j}_{Lt}+ \lD^i \Xi_{jt}\cX  H^{\alpha+j}_{Lt}\right)-\cX \lD^i H^\alpha_{Lt}+\]
\[+\sum_j[\lD^i,\cX ]\Xi_{jt}H^{\alpha+j}_{Lt}-[\lD^i,\cX ] H^\alpha_{Lt}\]

Using  expression (\ref{dxc}) for the  commutator,  we have:

\[\frac d{dt} \Delta^{\alpha i}_{Lt}=\cX \Delta^{\alpha i}_{Lt}+\sum_k\lD^i\xi_k\Delta^{\alpha k}_{Lt}.\]

Thus,   \(\Delta^{\alpha i}_{L0}  =   0\),  and we have,  by  Lemma  \ref{lemo5}, that \(\Delta^{\alpha i}_{Lt}=0\).
\end{proof}

The previous proposition  provides sufficient conditions
for the exponentiability of a \LB\ field of in \(\Je\infty\), however, an infinite number of them appear.

We now show that only  a finite subset of those conditions already is sufficient for  exponentiability.

\begin{theorem}\label{th4}
Let \(\cX\) be a \LB\ field in \(\Je\infty\)
Let \(A\) be a finite set of multipairs such that \(A \subset B\), defined in (\ref{bigb}).
Let \(k = 2|A|\), \(\wsp{\cX}k{}\) and \(\wsp {\cY}k{}\)
as defined by equalities (\ref{eq5} ) and (\ref{eq6}).

We assume that \(\wsp\Phi kt\), the flow of \(\wsp{\cX}k{}\) satisfies
\[\hbox{\rm(\(\star\))}\quad\left\{\begin{array}{l} \wsp {\cY}k{\ell}\wsp\Phi k{it}=0\quad i=1,2,\dots,n,\,\ell=1,\dots,\nu(k) \\
\wsp {\cY}k{\ell}\Phi^{(k)\alpha}_{Lt}=0 \quad L=1,2,\dots,m, \na\le k, \ell=1,\dots,\nu(k) \end{array}\right. .\]
Then \(\cX\) is exponentiable.
\end{theorem}

\begin{proof}
We first show that \(\ds \Phi^{(k)\beta}_{Mt} =F^\beta_M\left(\wsp \Phi k{it},\Phi^{(k)0}_{Mt}\right)\) for all
\((M, \beta), \nb \le k\). For this, it's enough that 
\[\sum_j\lD^i\wsp\Phi k{jt}\Phi^{(k)\alpha+j}_{Lt}=\lD^i\Phi^{(k)\alpha}_{Lt},\quad \forall i,\forall(L,\alpha) \,|\, \na\le k-1.\]

 Now, if \[\ds \Delta^{(k)\alpha i}_{Lt}=\sum_j\lD^i\wsp \Phi k{jt} \Phi^{(k)\alpha+j}_{Lt}-\lD^i\Phi^{(k)\alpha}_{Lt},\]
then
\[\frac d{dt}\Delta^{(k)\alpha i}_{Lt}=\sum_j\left(\lD^i\left(\wsp{\cX}k{}\wsp\Phi k{jt}\right)\Phi^{(k)\alpha+j}_{Lt} +\lD^i\wsp \Phi k{jt}\wsp \cX k{}\left(\Phi^{(k)\alpha+j}_{Lt}\right)\right)-\lD^i(\wsp{\cX}k{}\Phi^{(k)\alpha}_{Lt} )\]

Using  hypothesis (\(\star\)), one has
\[\frac d{dt}\Delta^{(k)\alpha i}_{Lt}=\sum_j\left(\lD^i\left(\cX \wsp\Phi k{jt}\right)\Phi^{(k)\alpha+j}_{Lt} +\lD^i\wsp \Phi k{jt}\cX  \left(\Phi^{(k)\alpha+j}_{Lt}\right)\right)-\lD^i(\cX \Phi^{(k)\alpha}_{Lt}) \]
and, for the commutator formula:
\[\frac d{dt}\Delta^{(k)\alpha i}_{Lt}=\cX \Delta^{(k)\alpha i}_{Lt}+\sum_j\lD^i\xi_i\Delta^{(k)\alpha j}_{Lt}\]

Using again hypothesis (\(\star\)), we can write the second
member above as
\[\wsp{\cX}k{}\Delta^{(k)\alpha i}_{Lt}+\sum_j\lD^i\xi_j\Delta^{(k)\alpha j}_{Lt}\]
With null initial conditions, Lemma 5 gives that \(\Delta^{(k)\alpha i}_{Lt}=0\)

Define now \(\Xi_{it}=\wsp \Phi k{it}\) and \(H^\alpha_{Lt}=\Phi^{(k)\alpha}_{Lt}\)
for  multipairs \((L,\alpha)\) such that \(\na \le k\).

For the rest of the multipairs let \(H^\beta_{Mt} =F^\beta_M(\Xi_{it},H^0_{Mt})\) and
\(\Phi_t=(\Xi_{it},H^\beta_{Mt})\).

Therefore, the following relations are valid

\[\hbox{\rm(\(\star\star\))}\quad\left\{
\begin{array}{l}
  H^0_{Lt}\circ \Phi_s =H^0_{Lt}(\Xi_{js},F^\beta_M(\Xi_{is},H^0_{Ms}))=\\
 =\Phi^{(k)0}_{Lt}(\wsp\Phi k{js},F^\beta(\wsp \Phi k {is},\Phi^{(k)0}_{Ms})) = \Phi^{(k)0}_{Lt}(\wsp\Phi k{js},\Phi^{(k)\beta}_{Ms})
   =\Phi^{(k)0}_{Lt+s}=H^0_{Lt+s},\\  \ \\
 \Xi_{it}\circ \Phi_s =\Xi_{it}(\Xi_{js}, F^\beta_M(\Xi_{is},H^0_{Ms}))= \\= \wsp \Phi k{it}(\wsp \Phi k {js},F^\beta(\wsp \Phi k {is},\Phi^{(k)0}_{Ms}))=
  \wsp \Phi k{it}(\wsp \Phi k {js},\Phi^{(k)\beta}_{Ms}) =\wsp \Phi k{it+s}=\Xi_{it+s}
\end{array}\right.\]

We now construct a local action of \(\lR\) in the set of germs of maps \(\cC^\infty(\lR^n,\lR^m)\).

For each function 	\(f: U\subset \lR^n 	\to		\lR^m\),  \(\Xi_{i0}\circ \tau^\infty f =\hbox{id}\),
therefore, for \(|t|\) sufficiently small and in a neighborhood \(V\) of \(x_0\) , \(\Xi_{it} \circ \tau^\infty f=\Psi_t\)
is invertible and \(H^0_{Lt}\circ \tau^\infty f\circ \Psi^{-1}_t\) is
defined in a neighborhood \(V\) of \(\Psi_t(x_0)\).

The map \(\theta_t: \{f_L\}_{x_0} \mapsto \{H^0_{Lt}\circ  \tau^\infty f\circ \Psi^{-1}_t\}_{\Psi_t(x_0)} = \{f_{Lt}\}_{\Psi_t(x_0)}\)
defines a curve in the space of germs. And, for each \(x \in V\),
\(H^0_{Lt}\circ  \tau^\infty f\circ \Psi^{-1}_t(\Psi_t(x)) =H^0_{Lt}\circ  \tau^\infty f(x)=f_{Lt}(x)\).
This means that \[\tau^0 f_{Lt}(x) = \{(\Xi_{it} \circ \tau^\infty f(x),H^0_{it}\circ \tau^\infty f(x)), x\in V\}.\]

If \(H^0_{Lt}\) and \(\Xi_{it}\)
satisfy the relations
in (\(\star\star\)), then \(\theta_t\)
is a local action and reciprocally:
\[\tau^0(\theta_t\circ \theta_sf(s))=\tau^0(\theta_t(\theta_sf(x)))=(\Xi_{it}(\tau^\infty(\theta_sf(x)),H^0_{Lt}(\tau^\infty(\theta_sf(x))) =\]
\[\left(\Xi_{it}(\Xi_{is},H^\alpha_{Ls})_{\tau^\infty f(x)},H^0_{Lt}(\Xi_{is},H^\alpha_{Ls})_{\tau^\infty f(x)}\right)=\]
\[=(\Xi_{i(t+s)},H^0_{L(t+s)})_{\tau^\infty f(x)}=\tau^\infty (\theta_{t+s}f(x)),\]
since
\[\pd{{}^{\na} f_{Lt}}{x^\alpha}(\Xi_{it}\circ \tau^\infty f(x))=H^\alpha_{Lt}\circ \tau^\infty f(x).\]

Consider the diagram below, where \(\theta_t(\{f\}_{x_0})=\{f_t\}_{\bar x_0}\) and \(\bar x_0=\Xi_{it}\circ \tau^\infty f(x_0)\) and \[\tau(\{f\}_{x_0})=\left(x_0, \pd{{}^{\na} f(x_0)}{x^\alpha}\right).\]

\begin{diagram}
\cG && \rTo^{\theta_t} && \cG \\
\dTo_\tau &&&& \dTo_\tau \\
\Je\infty &&\rFromshto^{\Phi_t} && \Je\infty \\
\end{diagram}

\(\Je\infty\)
can be identified with the quotient space
of \(\cG\) and \(\tau\) by the canonical map.

Differentiating the expression that defines \(\theta_t\),
we see that \(\tau(\theta_t\{f\}_{x_0})\) does not depend on the representative \(\{f\}_{x_0}\) of \(\tau(\{f\}_{x_0})\)
but on the  value \[\ds \tau f(x_0)=\left(x_0, \pd{{}^{\na} f(x_0)}{x^\alpha}\right).\]

As \(\theta_t\) is a local action,
\(\Phi_t\) is one also.
\end{proof}

We saw, in  Theorems \ref{th2} and \ref{th4},  proven above, that  conditions \(\wsp {\cY}k{\ell}\wsp \Phi k{it} = 0\)
and \(\wsp {\cY}k{\ell}\Phi^{(k)\alpha}_{Lt} = 0\)  are necessary and sufficient for the exponentiability of  \(\cX\). These conditions, however,
involve the resolution of a system of differentiable equations.

We wish to find, now, conditions on the components of
field \(\cX\) that do not depend on the  solution of a system of 
differential equations and that are sufficient for  the
hypotheses of Theorem \ref{th4} to be  satisfied, that is, for  Problem \ref{prob1}, below, to have solution. Besides this, if these
conditions are finite in number, we will have an effective criterion
to show that a given field is exponentiable.

Under certain hypotheses, we show, in the third part
of this work, a finite number of necessary and sufficient conditions for a particular case of a \LB\ field to be  exponentiable in \(\Jrr\).

\begin{problem}\label{prob1}
\[\left\{\begin{array}{ll}
\ds\frac d{dt}\wsp \Phi k{it}=\wsp{\cX}k{}\wsp \Phi k{it},&\quad i=1,2,\dots,n,\\ \\
\ds\frac d{dt} \Phi^{(k)\alpha}_{Lt}=\wsp{\cX}k{} \Phi^{(k)\alpha}_{Lt},&\quad L=1,2,\dots,m;\, \na\le k,\\ \\
\wsp {\cY}k{\ell}\wsp \Phi k{it}=0,&\quad i=1,2,\dots,n;\, \ell=1,2,\dots,\nu(k),\\ \\
\wsp {\cY}k{\ell} \Phi^{(k)\alpha}_{Lt}=0,&\quad L=1,2,\dots,m;\, \na\le k;\, \ell=1,2,\dots,\nu(k),\\ \\
\wsp \Phi k{i0}=x_i,&\quad i=1,2,\dots,n,\\ \\
\Phi^{(k)\alpha}_{L0}=\ya_L,&\quad L=1,2,\dots,m;\, \na\le k.
\end{array}
\right.\]
\end{problem}

 We rewrite  Problem \ref{prob1} in the following form:

\begin{problem}\label{prob2}
\[\left\{\begin{array}{ll}
\ds\frac d{dt}f_t=\cX f_t,\\ \\
{\cY}_pf_t =0,\quad p=1.2.\dots,\nu \\ \\
f_0=g, \hbox{\rm\ where\ } g \hbox{\rm \ is a given function.}
\end{array}\right.\]
\end{problem}

We observe that the symbols \(\cX\) and \(\cY_p\) used here
symbolize \(\wsp{\cX}k{}\) and \(\wsp {\cY}k{\ell}\)
respectively,  vector fields
 in spaces of finite dimension, say in \(U\subset \lR^k\), for a convenient \(k\).

Suppose there is a solution \(f_t\) of Problem \ref{prob2}.
Let \(\cL\)
be the set of fields \(\cZ\) such that for all \(t\), \(\cZ f_t = 0\).

\(\cL\) is a Lie algebra (and a module over \(\cC^\infty(U)\))
invariant under \([\cX ,\cdot]\).

Indeed \(\cZ_1 ,{\cZ}_2\in \cL \Rightarrow [{\cZ}_1,{\cZ}_2] \in \cL\), and,
\[\ds [\cX ,{\cZ}] f_t=\cX {\cZ} f_t-{\cZ}\cX  f_t=0-{\cZ}\frac d{dt} f_t =-\frac d{dt}{\cZ}f_t =0.\]

Thus, if there is a solution \(f_t\) of Problem \ref{prob2} above, then \(\cZ f_t=0\) for all \(\cZ\) belonging to the \(L\), where \(L\) is the smallest
algebra of Lie that contains \(\cY_p\) and is invariant under \([\cX ,\cdot]\).

Starting from this we ask  if  condition \(\cZ f_0=0\) for all \(\cZ \in L\), would be sufficient for the existence of a solution \(f_t\)
of the given  problem.

We consider, however, the following example for \(K = 1\):
\(\ds \cX =\frac d{dx}\) and \(\ds {\cY} = \psi \frac d{dx}\)
where \(\psi\) has the following graph:

\vskip 24pt
\includegraphics[scale=0.4]{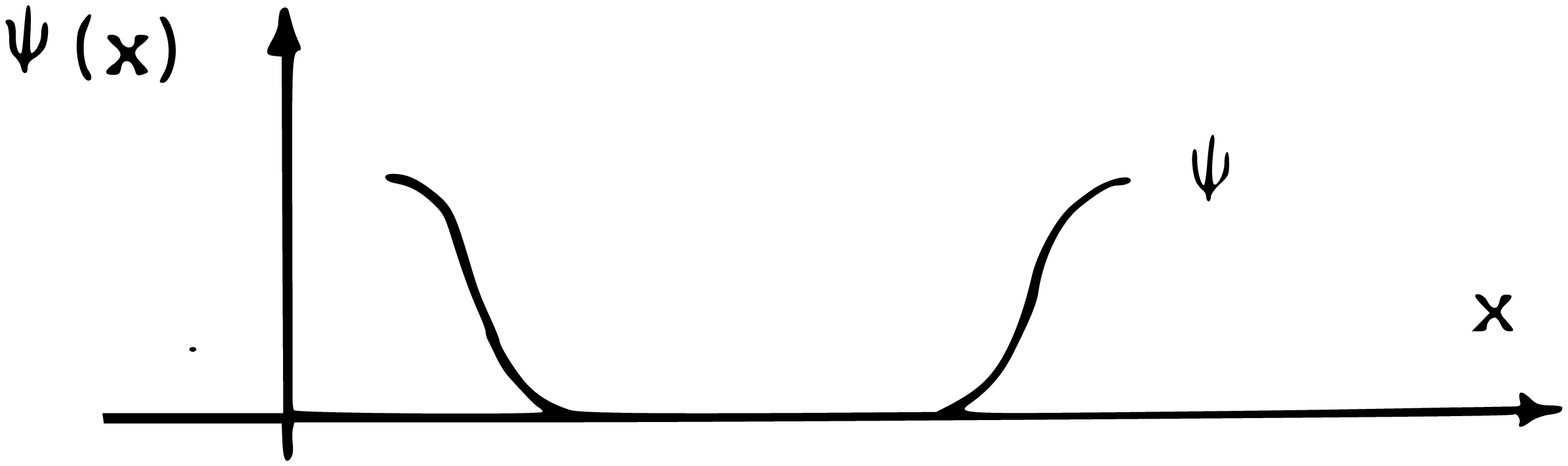}
\vskip 2ex

In this case, \(L\), the smallest Lie algebra  generated by \(\cY\) and
and invariant under \([\cX , \cdot]\), consists of elements of type
\(\ds \alpha \frac d{dx}\)
where \(\alpha\) is a polynomial in \(\psi\) and its derivatives. For example:
\[ [\cX,\cY]=\psi'\frac d{dx},\quad [\cX , [\cX ,{\cY}]=(\psi''\psi-(\psi')^2) \frac d{dx},\dots\]

Let \(\Phi_t (x)\) be the flow of \(\cX\);  \(\Phi_t (x)=x+t\) and \[\left\{\begin{array}{l}\ds \frac d{dt}f_t=\cX f_t\\
\hskip 70pt\Rightarrow f_t=g(x+t).\\
f_0 = g
\end{array}\right.
\]
Assume \(g\) has the following graph:
\vskip24pt

\begin{center}
\includegraphics[scale=0.4]{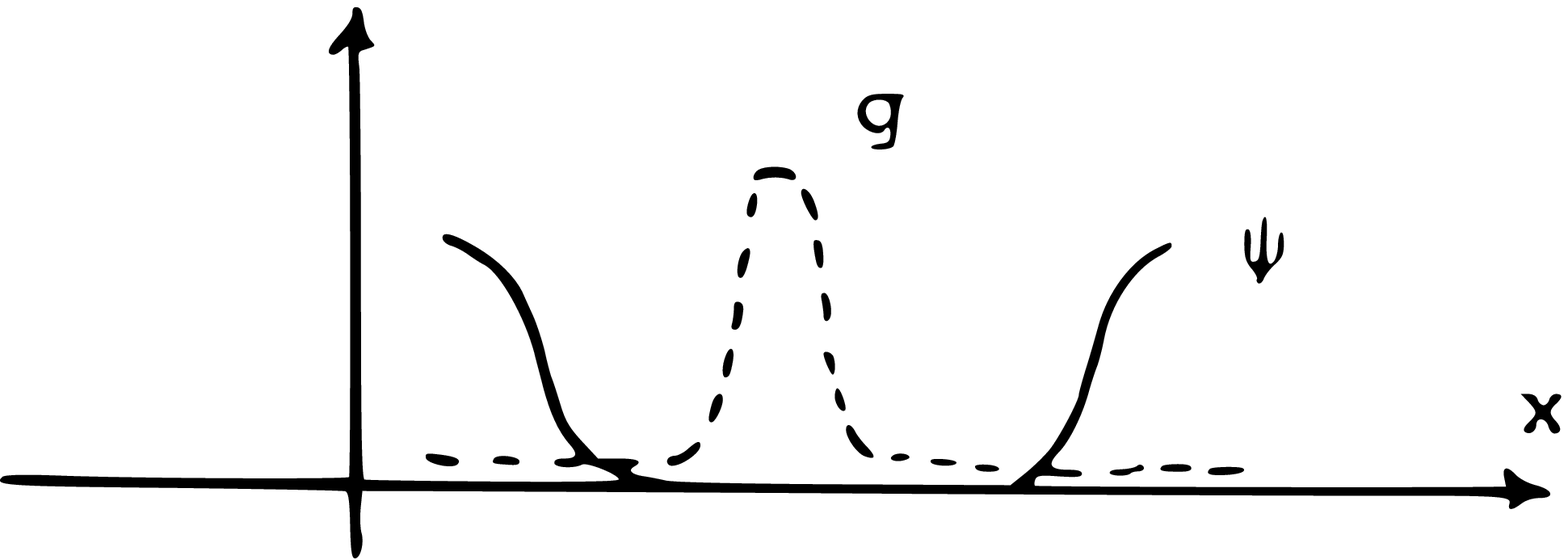}
\end{center}

Here \(\cZ g=0\) for all \(\cZ\in L\) but \(\cZ f_t \neq 0\).

Therefore the fact that \(\cZ f_0\) is null for all element of \(L\)
does not  permit, in general, to conclude \(\cZ f_t\) is null
for all \(t\) and all  \(\cZ\) in \(L\).

To guarantee this we have to impose other restrictions on the
fields \(\cX\) and \(\cY_p\), which we do in the theorems that follow.

\begin{tha}\label{th5a}
Let \(V\) be an open set of \(\lR^k\)
and \(L\) the smallest Lie algebra of fields in \(V\) that contains \(\cY_p\)
and is invariant under \([\cX , \cdot]\).
Suppose that \(L\) defines a distribution of rank \(\ell\le k\)
in a neighborhood \(W\) of \(p\in \lR^k\) and that \(\cZ g = 0\) for all \(\cZ\in L\). Then  Problem \ref{prob2} has solution.
\end{tha}

\begin{proof}
By Frobenius's theorem, there is a system of coordinates
\((U, z^1, \dots ,z^k )\) in a neighborhood \(U\) of \(p\)
such that each section determined by \(z^{\ell+1}=c_1,\dots, z^k=c_{k-\ell}\)
is an integral  manifold of L.

Calling \(\alpha_1(z) = z^{\ell+1},\dots,\alpha_{k-\ell}(z) = z^k\),
one verifies that \(\cX \alpha_i = q_i (\alpha_1,\dots,\alpha_{k-\ell})\)
(Lemma 6), and we have \(\cZ f = 0,\, \forall\, {\cZ} \in L\Rightarrow f=\phi(\alpha_1,\dots,\alpha_{k-\ell})\).

Thus condition \(\cZ f_t = 0,\, \forall \,{\cZ}\in L\) can be rewritten as
\(f_t=\Phi_t(\alpha_1,\dots,\alpha_{k-\ell})\) and we have:
\[\pd{}tf_t=\cX f_t\Leftrightarrow \pd{}t \Phi_t= \sum_{j=1}^{k-\ell} \pd{\Phi_t}{\alpha_j} (\alpha_1,\dots,\alpha_{k-\ell})\cX \alpha_j =\]
\[=\sum_{j=1}^{k-\ell} q_j (\alpha_1,\dots,\alpha_{k-\ell})\pd{\Phi_t}{\alpha_j}.\]

If the initial given condition \(g\) is such that \(\cZ g = 0\) for all \(\cZ\in L, g = \gamma(\alpha_1,\dots,\alpha_{k-\ell})\)

Then, by Lemma \ref{lemo4}, we have that \[\left\{\begin{array}{l}
\ds\pd{}t \Phi_t=\sum_{j=1}^{k-\ell}q_j(\alpha_1,\dots,\alpha_{k-\ell})\pd{\Phi_t}{\alpha_j}\\ \\
\Phi_0=\gamma(\alpha_1,\dots,\alpha_{k-\ell})
\end{array}\right.\]
has a solution.
\end{proof}

Consider again Problem \ref{prob2} above and suppose that it
has a solution \(f_t\).

The Lie algebra of vector fields \(\cZ\) such that \(\cZ f_t =0\) for all \(t\) is also a module over
\(C^\infty (U)\),
\(U\) open in \(\lR^k\).

Let V be the smallest module of vector fields on \(C^\infty (U)\)
containing \(\cY_p\)
and invariant under \(\hbox{ad}_\cX \)

Therefore, \(\cZ f_t =0\) for all \(\cZ \in V\).

Using the  previous counterexample and analogous arguments
to those used there, we see the necessity of imposing 
additional restrictions to conclude that \(\cZ f_t =0\) for all \(t\) and all \(\cZ\in V\),
besides \(\cZ f_0 = 0\) for  all \(\cZ\in V\).

\begin{thb}\label{th5b}
Let \(V\) be the smallest module of vector fields over
\(\cC^\infty (U)\), \(U\) open in \(\lR^k\),
that contains \(\cY_p\)
and is invariant under \(\adx\).
Suppose that \(V\) is finitely generated over \(\cC^\infty (U)\) and that \(\cZ g = 0\) for all \(\cZ\in V\).

Then  Problem \ref{prob2} has a solution.
\end{thb}

\begin{proof}
If \(V\) is finitely generated over
\(\cC^\infty (U)\), that is, there are
\(\cC_1,\dots,\cC_n \in V\) such that: \(\ds {\cZ}\in V\Rightarrow {\cZ}=\sum_{\hbox{fin}} g_i\cC_i\)
where \(g_i \in \cC^\infty (U)\)
we wish to show therefore that \[\cC_if_0=0 \Rightarrow {\cZ}(f_0\circ \Phi_t)=0 \]
for all \(\cZ\in V\), \(|t|\) sufficiently small.
\[\left\{\begin{array}{l}
\ds\ddt f_t =\cX  f_t \\ \\
\cC_if_t=\cC_i(f_0\circ \Phi_t)
\end{array}\right.\]

\[\ddt \cC_if_t =\cC_i \ddt f_t =\cC_i\cX f_t -[\cX ,\cC_i]f_t =\]
\[=\cX \cC_if_t-\sum_j g_{ij}\cC_jf_t \quad \hbox{(by Lemma 5)} \]
\[\cC_if_t=0\Rightarrow {\cZ}f_t=0.\]
\end{proof}

\begin{thc}\label{th5c}
Let \(\cX\) be an analytic field and \(W\) the smallest space
of vector fields containing \(\cY_p\) and is invariant under \([\cX ,\cdot]\).

If \(\cZ f_0 =0\), \(\forall {\cZ}\in W\), then \(\cZ f_t = 0\), \(\forall t\).
\end{thc}

\begin{proof}
\(\cX\) analytic \(\ds \Rightarrow f_t=(f_0\circ\Phi_t) = f_0+t\cX f_0+\cdots+\frac{t^n}{n!}\cX ^nf_0+\cdots\)
\[\pd{}{t}{\cZ}f_t ={\cZ}f_0+t{\cZ}\cX f_0+\cdots+\frac{t^n}{n!}{\cZ}\cX ^nf_0+\cdots=\]
\[\hbox{(by Lemma \ref{lemo7})}\, = {\cZ}f_0+t\{\cX {\cZ}-[\cX ,{\cZ}]\}f_0 +\cdots+\frac{t^n}{n!}\{\cX ^n{\cZ}+\cdots\}f_0=0.\]
\end{proof}

\section{Concrete example}\label{sec:terc}

Consider the fields \(\cX\) and \(\cY_p, p = l,2, \dots ,\nu\),
introduced in the observations that preceded  theorems  \hyperref[th5a]{7(A)} and \hyperref[th5b]{7(B)}.
It's easy to see that the set of vector fields
\(\{\adx^q {\cY}_p , p = 1,2, ... ,\nu;\, q=0,1,2, ... \}\) generates \(W\),
the smallest vector space that contains  \(\cY_p,\, p=1,2, ... ,\nu\)
and that is invariant under \(\adx = [\cX , \cdot]\).  The same set of fields also generates
\(V\), the smallest module of fields over \(\cC^\infty (U)\), \(U \subset \lR^k\),
that contains  \(\cY_p\) and that is invariant under \(\adx\).

On the other hand, when \(t = 0\),
\(\Xi_{it}=x_i\), \(i = 1, \dots ,n\)
and \(H^\alpha_{Lt}=y^\alpha_1\).
It is therefore necessary that \(\cZ x_i =0\,\, i=1,2, \dots ,n\)
and \(\cZ y^\alpha_L=0, \, (L,\alpha)\in A\)
(A defined in Theorem \ref{th2}), for all \(\cZ\in W\)
(respectively, all \(\cZ\in V\)),
for the  existence of a solution
to Problem \ref{prob1} and, therefore, for the exponentiability of the related \LB\ field.

In particular, setting \(\cZ = \adx^q {\cY}_p\) one verifies that a
set of necessary conditions for the exponentiability of a \LB\ field
related to Problem \ref{prob1}, should always include
the following: the coefficients of \(\pd {}{x_i}\) and \(\pd {}{y^\alpha_L}\), 
\(i=1,2, \dots ,n\) and \((L,\alpha)\in A\),
should be zero in the expressions
of \(\adx^q {\cY}_p, p = 1,2, \dots  , \nu ; q= 0,1,2, ...\)

Beyond this, if for some integer \(\tilde N\) and for all
\(p = 1,2, \dots ,\nu\), \(\adx^{\tilde N} {\cY}_p\) belongs to the module generated by
\(\adx^j {\cY}_p\),   \(0 \le j < \tilde N\),
then \(V\) is finitely generated.

We now present, an example where we show that the
minimal necessary conditions in the above sense are,
under certain hypotheses, sufficient for the exponentiability of a \LB\  field  \(\cX\) in, \(\Jrr\).

Consider a \LB\ field \(\cX\) in \(\Jrr\),
\[\cX = \xi\pd{}x+\sum_{(L,\alpha)}\eta^\alpha_L\pd{}{y^\alpha_L}\]
such that \(\xi, \eta^0_1, \eta^0_2\)
does not depend on \(\yg_M\), \(\gamma\ge 2\)

In this section we adopt the convention that repeated indices are
summed and we omit the summation symbol.

We shall also use the following notation:

\[O_L=\pd{\eta^0_L}x -y^1_L\pd \xi{x},\quad L=1,2.\]
\[N^Q_L=\pd{\eta^0_L}{y^0_Q} -y^1_L\pd \xi{y^0_Q},\quad L,Q=1,2.\]
\[M^Q_L=\pd{\eta^0_L}{y^1_Q} -y^1_L\pd \xi{y^1_Q},\quad L,Q=1,2.\]
\[Q^R_M=\xi\pd{M^R_M}x +\eta^0_K\pd{M^R_M}{y^0_K} +(O_K+y^1_QN^Q_K)\pd{M^R_M}{y^1_K}+\]
\[-\left(M^R_N\pd{O_M}{y^1_N}+M^R_QN^Q_M+M^R_Ny^1_Q\pd{N^Q_M}{y^1_N}\right).\]

\stepcounter{theorem}
\begin{theorem}\label{th6}
Let \(\cX\) be a \LB\ field  
in \(\Jrr\).
For  \(\cX\) to be exponentiable to a  flow \(\Phi_t=(\Xi_t, H^\alpha_{Lt})\)
such that \(\Xi_t\) and \(H^0_{Lt}\) do not depend on \(\yg_M, \gamma\ge2, M=1,2 \)
it is necessary that:
\[\hskip 3em \hbox{(i)}\quad M^R_M\pd\xi{y^1_M}=0;\quad M^R_M\pd{\eta^0_L}{y^1_M}=0,\quad L,R=1,2 \hbox{\ \ and\ } \]
\[\hbox{(ii)}\quad Q^R_M\pd\xi{y^1_M}=0;\quad Q^R_M\pd{\eta^0_L}{y^1_M}=0,\quad L,R=1,2  \]
\end{theorem}

We note that (i) \(\Rightarrow M^2=0\).

\begin{proof}
As \(\cX\) is \LB,  we have that \(\eta^1_L=O_L+y^1_QN^Q_L+y^2_QM^Q_L\),
where, it is easy to see that \(O_L\), \(N^Q_L\), \(M^Q_L\)
do not depend on \(\yg_M\) if \(\gamma\ge 2\).

Using  definitions \ref{def5} and \ref{def6} of Section \ref{sec:segu}, we have that:
\[\wsp{\cX}1{}= \xi\pd{}x+\eta^0_L\pd{}{y^0_L}+(O_L+y^1_QN^Q_L)\pd{}{y^1_1}\]
and \(\ds \wsp {\cY}1R=M^R_M\pd{}{y^1_M}\),
are fields in \(\Jrr\).

As \[ \left[\wsp{\cX}1{},\wsp {\cY}1R\right]=-M^R_M\left(\pd{\xi}{{\cY}^1_M}\pd{}x +\pd{\eta^0_L}{y^1_M}\pd{}{y^0_L}\right) +Q^R_L\pd{}{y^1_L}\]
is a element of \(V\), we conclude (i).

Beyond this, \[ \left[\wsp{\cX}1{} \left[\wsp{\cX}1{},\wsp {\cY}1R\right]\right]=Q^R_M\left(\pd{\xi}{{\cY}^1_M}\pd{}x +\pd{\eta^0_L}{y^1_M}\pd{}{y^0_L}\right) +R^R_M\pd{}{y^1_M}\]
where \(R^R_M\) is a convenient function, is also an element of \(V\).

Thus, we have (ii) and the theorem is proved.
\end{proof}

\begin{theorem}\label{th7}
Consider a \LB\ field in \(\Jrr\),
\[\ds \cX =\xi\pd{}x+\sum_{(\alpha,L)}\eta^\alpha_L\pd{}{\ya_L}.\]

If \(M \neq 0\) at the point
\(P=(\bar x,\bar y^0_1,\bar y^0_2,\bar y^1_1,\bar y^1_2)\) in \(\Jrr\)
and  conditions (i) and (ii)
of the previous theorem are satisfied, then \(\cX\) is exponentiable to a
local flow \(\Phi_t=(\Xi_t,H^\alpha_{Lt})\),
that is defined on a set in \(\Jrr\)
that  projects to a neighborhood of \(P\), and its components \(\Xi_t\), \(H^0_{Lt}\)
do not depend on \(\yg_M\), \(\gamma \ge 2,\, M=1,2\).
\end{theorem}

We first show a fact about \(2 \times 2\) matrices
that we use in the proof of this theorem: given \(2 \times 2\) matrices \(A\) and \(M\)
such that \(M^2 = 0\), \(M\neq 0\) and \(AM=0\), then there exists
a \(2 \times 2\) matrix \(B\),
such that \(A = BM\).
If \(M\) is a \(\cC^\infty\) matrix function,
then it is possible to chose \(B\) locally \(\cC^\infty\).
To see
this, call \(A_i\) and \(M_i\)
the  \(i\)-th line of \(A\) and \(M\) respectively and \(M^j\)
a \(j\)-th column of \(M\).
If \(M^2=0\), then \(\langle M_i, M^j\rangle = 0\), \(i,j = 1,2\).
As \(M\neq 0\), the space \(V\) generated by
\(M_1\) and \(M_2\) has dimension l. Now, \(AM=0\Rightarrow \langle A_i,M^j \rangle=0,\,
i,j =1,2\Rightarrow A_i \in V\).
If \(M_1\neq 0\), there are \(b_i, \,i=1,2\) such that \(A_i^j=b_iM^j_1\). Then \(\ds b_i=\frac{A^j_i}{M^j_1}\) where \(j\) is such that \(M^j_1\neq 0\),
therefore, in the case of \(M\) being a matrix function, \(b_i\) is locally \(\cC^\infty\).
Setting \(\ds B=\left(\begin{array}{cc} b_1 & 0 \\ b_2 & 0\end{array}\right)\),
we have \(A = BM\). In the other case,
if \(M_1=0\), then \(M_2\neq 0\)
and we proceed analogously to show that \(A = BM\),
where \(\ds B=\left(\begin{array}{cc} 0& b_1 \\ 0& b_2\end{array}\right)\).

\begin{proof}
  (of Theorem \ref{th7}) 

We shall use Theorems \ref{th4} and \hyperref[th5b]{7(B)} to show that these
conditions imply the existence of a flow  \(\Phi_t=(\Xi_t,H^\alpha_{Lt})\)
 such that \(\Xi_t\) and \(H^\alpha_{Lt}\)
do not depend on \(\yg_M\), \(\gamma \ge 3,\, M=1,2\).

Let the fields be given by definitions \ref{def5} and \ref{def6}:
\[\wsp{\cX}2{}= \xi\pd{}x+ \eta^0_L\pd{}{y^0_L}+\left(O_L+y^1_QN^Q_L+y^2_LM^Q_L\right)\pd{}{y^1_L}+B_L\pd{}{y^2_L}\]
where \(B_L\) does not depend on \(\yg_M\), \(\gamma \ge 3,\, M=1,2\) and \(\ds \wsp {\cY}2R=M^R_P\pd{}{y^2_P}\).

Now,
\[[\wsp{\cX}2{},\wsp {\cY}2R]=\xi\pd{M^R_P}x\pd{}{y^2_P}+\eta^0_L\pd{M^R_P}{y^0_L}\pd{}{y^2_P}+\]
\[+\left(O_L+y^1_QN^Q_L+y^2_LM^Q_L\right)\pd{M^R_P}{y^1_L}\pd{}{y^2_P} +\left(-M^R_S\pd{O_P}{y^1_S} - M^R_QN^Q_P+\right.\]
\[-\left.y^1_QM^R_S\pd{N^Q_p}{y^1_S}-M^R_Q\lD M^Q_P -y^2_QM^R_S\pd{M^Q_P}{y^1_S}+M^R_P\lD \xi\right)\pd{}{y^2_P}.\]

Observe that, given \(M^2=0\), \(M \partial\, M=-\partial M\, M\) for any
derivation \(\partial\).

Observe again that \[\pd{}{y^1_K}M^I_J=\pd{}{y^1_I}M^K_J+\delta_{I,J}\pd\xi{y^1_K}-\delta_{K,J}\pd\xi{y^1_I}\]
for all \(I,J,K = 1,2\). We  use the two observations  to see that the fifth
and the tenth terms of the expression for \([\wsp{\cX}2{},\wsp {\cY}2R]\)
are, respectively, equal to \[-y^2_Q\left(\pd{M^Q_L}{y^1_R}M^L_P+\pd\xi{y^1_R}M^Q_P\right)\pd{}{y^2_P}\]
and
\[y^2_Q\left(\pd{M^R_L}{y^1_Q}M^S_P+\pd\xi{y^1_Q}M^R_P\right)\pd{}{y^2_P}\]
and whose sum shall be written
as \(\ds B^R_NM^N_P\pd{}{y^2_P}\). 

Note also that the sum of the first through the
fourth, and the sixth through the eighth terms constitute \(\ds Q^P_R\pd{}{y^2_P}\)

As \(M^2 = 0\), \(M\neq 0\) and \(Q M =0\), we have \(Q^P_R=P^R_NM^N_P\).

Therefore, \(\ds [\cX^{(2)},\cY^{(2)}_R ] =A^R_NM^N_P\pd{}{y^2_P}\) where \(A^R_N\)
is the total of the coefficients of \(\ds M^N_P\pd{}{y^2_P}\) in the expression above.

Thus, \([\cX ^{(2)},\cY^{(2)}_R ]=A^R_N\wsp {\cY}2N\),
which  proves that the module \(V\) is finitely generated.

By Theorem \hyperref[th5b]{7(B)}, we see that \(\wsp{\cX}2{}\) and \(\wsp {\cY}2R\)
satisfy the hypotheses of Theorem \ref{th4}, thus \(\cX\) is exponentiable and its flow \(\Phi_t\) is such that \(\Xi_t=\wsp \Phi2{1t}\) and \(H^0_{Lt}=\Phi^{(2)0}_{Lt}\).
Thus, we can now affirm that \(\Xi_t\) and \(H^0_{Lt}\) do not depend
on \(\yg_M,\,\gamma \ge 3\).

We now show  that \(\Xi_t\) and \(H^0_{Lt}\) do not depend
on \(y^2_M,\,M=1,2\),
examining
\(\ds \ddt\left(\pd{}{y^2_M}\Phi^{(2)0}_{Lt}\right)\) and \(\ds \ddt\left(\pd{}{y^2_M}\Phi^{(2)}_{Lt}\right)\).

We see:

\begin{eqnarray}\nonumber
 &&\ddt\left(\pd{}{y^2_M}\Phi^{(2)0}_{Lt}\right)=\pd{}{y^2_M}(\wsp{\cX}2{}\Phi^{(2)0}_{Lt}) = \\ \nonumber
 &&=\wsp{\cX}2{}\left(\pd{}{y^2_M}\Phi^{(2)0}_{Lt}\right)-\left[\wsp{\cX}2{},\pd{}{y^2_M}\right]\Phi^{(2)0}_{Lt} =\\ \label{star}
 &&=\wsp{\cX}2{}\left(\pd{}{y^2_M}\Phi^{(2)0}_{Lt}\right)+M^M_Q\pd{}{y^1_Q}\Phi^{(2)0}_{Lt}+ D^M_Q\left(\pd{}{y^2_Q}\Phi^{(2)0}_{Lt}\right).
\end{eqnarray}
Calculating 

\(\ds \ddt\left(M^M_Q\pd{}{y^1_Q}\Phi^{(2)0}_{Lt}\right)\),
where \(\ds M^M_Q\pd{}{y^1_Q}\), as we already saw, is \(\wsp {\cY}1M\), and \[\ddt (\wsp {\cY}1M \Phi^{(2)0}_{Lt})= \wsp {\cY}1M\wsp{\cX}2{}\Phi^{(2)0}_{Lt}=\]
\[=\wsp{\cX}2{}\wsp {\cY}1M \Phi^{(2)0}_{Lt}+[\wsp {\cY}1M, \wsp{\cX}2{}]\Phi^{(2)0}_{Lt}.\]

We write \(\wsp{\cX}2{}=\wsp{\cX}1{}+(\wsp{\cX}2{}-\wsp{\cX}1{})\) and
\[\wsp{\cX}2{}-\wsp{\cX}1{}=y^2_QM^Q_K\pd{}{y^1_K}+S^R\pd{}{y^2_R}=y^2_Q\wsp {\cY}1Q+S^R\pd{}{y^2_R}\]

Thus, for \([\wsp {\cY}1M, \wsp{\cX}2{}]\) one has:
\[[\wsp {\cY}1M, \wsp{\cX}1{}]+y^2_Q[\wsp {\cY}1M, \wsp {\cY}1Q]+\wsp {\cY}1M(S^R)\pd{}{y^2_R}.\]

As \([\wsp {\cY}1M, \wsp {\cY}1Q]= T_{M,Q}\wsp {\cY}1P\),
we have that

\begin{eqnarray}\nonumber
&&\ddt (\wsp {\cY}1M \Phi^{(2)0}_{Lt})=\wsp{\cX}2{}\wsp {\cY}1M\Phi^{(2)0}_{Lt}-P^N_M(y^1_N\Phi^{(2)0}_{Lt})+ \\ \label{starstar}
&&+P^Q_M{\cY}^2_Q(\wsp {\cY}1P\Phi^{(2)0}_{Lt})+\wsp {\cY}1M(S^R)\pd{}{y^2_R}\Phi^{(2)0}_{Lt}
\end{eqnarray}

Combining  equations obtained in (\ref{star}) and (\ref{starstar}) in a system
and remembering that the initial conditions are null, we have, by
Lemma \ref{lemo5}, that \[\pd{}{y^2_M}\Phi^{(2)0}_{Lt}=0\hbox{\ and\ }\wsp {\cY}1M\Phi^{(2)0}_{Lt}=0.\]

Proceeding analogously for \(\wsp \Phi2{1t}\), we have \(\wsp \Phi2{1t}=\wsp \Phi1{1t}=\Xi_{t}\) and 
\(\Phi^{(2)0}_{Lt}=\Phi^{(1)0}_{Lt}=H^0_{Lt}\), which shows that the
necessary conditions are also sufficient for the existence of  \(\Phi_t\) with \(\Xi_{it}\) and \(H^0_{it}\) independent of \(\yg_M, \, \gamma \ge 2\).
\end{proof}

Thus, the flow \(\wsp \Phi1t\) is obtained through the following
system of ordinary differentiable equations:
\[\dot x = \xi(x, y^0_M,y^1_M);\quad L,M = 1,2\]
\[\dot y^0_L = \eta^0_L(x, y^0_M,y^1_M);\quad L,M = 1,2\]
\[\dot y^1_L = O_L +y^1_QN^Q_L;\quad L,M = 1,2\]

Using the previous theorems, we have \(\Xi_{it}=\wsp \Phi1{it}\) and \(H^0_{Lt}=\Phi^{(1)0}_{it}\).
The rest of the components of \(\Phi_t\) are obtained starting with these and  the differential operators \(F^\beta_M\).

Particularly, we suppose that \(\xi\) and \(\eta^0_L\), \(L=1,2\),
do not depend on \(x\) nor on \(y^0_M,\, M = 1,2\).
We continue to suppose that \(M\neq0\),
as in Theorem \ref{th7}.

The problem now presents interesting geometry.

The system of differentiable equations that we have
to resolver to find the flow \(\wsp \Phi1t\) is reduced to \[\dot x = \xi(y^1_1,y^1_2)\]
\[\dot y^0_L = \eta^0_L(y^1_1,y^1_2);\quad L=1,2\]
\[\dot y^1_L = 0;\quad L=1,2\]

Thus, \(\wsp\Phi1t(x,y^0_L,y^1_M)=(x+t\xi(y^1_1,y^1_2), y^0_L+t\eta^0_L(y^1_1,y^1_2), y^1_M);\,L,M=1,2\)

If \(\cX\) is exponentiable to \(\Phi_t=(\Xi_t,H^\alpha_{Lt})\), we have shown already that
\(\Xi_t=x+t\xi(y^1_1,y^1_2)\) and \(H^0_{Lt}=y^0_L+t\eta^0_L(y^1_1,y^1_2)\); \(L=1,2\)
and \(H^1_{Lt}\) satisfy \(\lD\Xi_tH^1_{Lt}=\lD H^0_{Lt}\),
since \(\cX\) is \LB. Thus
\newcommand\YD[1]{1+{#1}\left(y^2_Q\pd{\xi}{y^1_Q}\right)}
\newcommand\YF[2]{\frac{y^1_{#2}+{#1}\left(y^2_Q\pd{\eta^0_L}{y^1_Q}\right)}{\YD {#1}}}

\[H^1_{Lt}=\YF tL.\]

From exponentiability, we have that for each \(L =1,2\),

\[H^0_{Lt}\circ \Phi_s=H^0_{Lt}\left(x+t\xi, y^0_K+s\eta^0_K, \YF sK,\dots\right)_{K=1,2}=\]
\[=y^0_L+s\eta^0_L(y^1_1,y^1_2)+t\eta^0_L\left(\YF sK\right)_{K=1,2} =\]
\[H^0_{L(s+t)}=y^0_L+(s+t)\eta^0_L(y^1_1,y^1_2).\]

Therefore, defining \(\ds \nabla \psi =\left(\pd \psi{y^1_1}, \pd\psi{y^1_2}\right)\) for any function \(\psi(y^1_1,y^1_2)\),
and defining \(v=s(y^2_1,y^2_2)\), we have:
\[\eta^0_L(y^1_1,y^1_2)\stackrel{*}{=}\eta^0_L\left(y^1_1+\frac{\langle v,\nabla \eta^0_1-y^1_1\nabla \xi\rangle}{\YD s},
y^1_2+\frac{\langle v,\nabla \eta^0_2-y^1_2\nabla \xi\rangle}{\YD s}\right).\]

Consider \(W\) the manifold generated by

\[\left(\frac{\langle v,\nabla \eta^0_1-y^1_1\nabla \xi\rangle}{\YD s},
\frac{\langle v,\nabla \eta^0_2-y^1_2\nabla \xi\rangle}{\YD s}\right)\quad\]
for \((y^1_1,y^1_2)\)  in a neighborhood of \((0,0)\) and \(s\) in a neighborhood of \(0\).

If \(W\) is bidimensional,  equality (\ref{star}) shows that \(\eta^0_L\)
is constant in a neighborhood of \((p,q)\), contradicting the bidimensionality.
\(W\) can not be a point, since \(M\neq0\) in \((p,q)\), by hypothesis. \(W\) can only be  unidimensional, and by its expression, we see that must be a line.

Then, by (\ref{star}), \(\eta^0_L\), \(L=1,2\)
are constants along a
 family of lines in a neighborhood of \((p,q)\) in the variables \(y^1_1,y^1_2\).
The same happens with \(\xi(y^1_1,y^1_2)\).

We assume that this family of lines is parameterized by \(\lambda\).
\[\eta^0_1=f_1(\lambda(y^1_1,y^1_2)),\, \eta^0_2=f_2(\lambda(y^1_1,y^1_2))\hbox{\ and\ }\xi=g(\lambda(y^1_1,y^1_2)), \]
\[\nabla \eta^0_1=f'_1\nabla\lambda,\, \nabla \eta^0_2=f'_2\nabla\lambda\hbox{\ and\ }\nabla \xi=g'\nabla\lambda\]

We have then

\begin{eqnarray}\nonumber
 \nabla \eta^0_1-y^1_1\nabla \xi &=& f'_1\nabla\lambda -y^1_1g'\nabla\lambda \\
\nabla \eta^0_2-y^1_2\nabla \xi &=& f'_2\nabla\lambda -y^1_2g'\nabla\lambda.
\end{eqnarray}

We defines \[ m(\lambda)=\frac{\langle v,\eta^0_2-y^1_2\nabla \xi\rangle}{\langle v,\eta^0_1-y^1_1\nabla \xi\rangle} = \frac{f'_2 -y^1_2g'}{f'_1 -y^1_1g'}\]
that is, the angular coefficient of the lines in a neighborhood where  there are no vertical segments
(in a neighborhood of a vertical segment, we can use \(1 /{m(\lambda)}\)).

We conclude that:

\[f'_2 - y^1_2 g' = m(\lambda) (f'_1-y^1_1 g')\quad\hbox{or}\]
\[f'_2-m(\lambda)f'_1=(y^1_2-m(\lambda)y^1_1)g'\]

As \(y^1_2-m(\lambda)y^1_1\) is constant
along these lines, say, equal to \(q_0(\lambda)\), we have
\[f'_2-m(\lambda)f'_1(\lambda)=q_0(\lambda)g'(\lambda).\]

Reciprocally, given one foliation of a an open set in the \(y^1_1,\,y^1_2\) plane
by a family of segments of lines parameterized by
\(\lambda\) (see figure below), each solution \((f_1,f_2,g)\) of the  differential equation
 \(f'_2-m(\lambda)f'_1(\lambda)=q_0(\lambda)g'(\lambda)\), being \(m(\lambda)\)

\begin{center}
\includegraphics[scale=0.5]{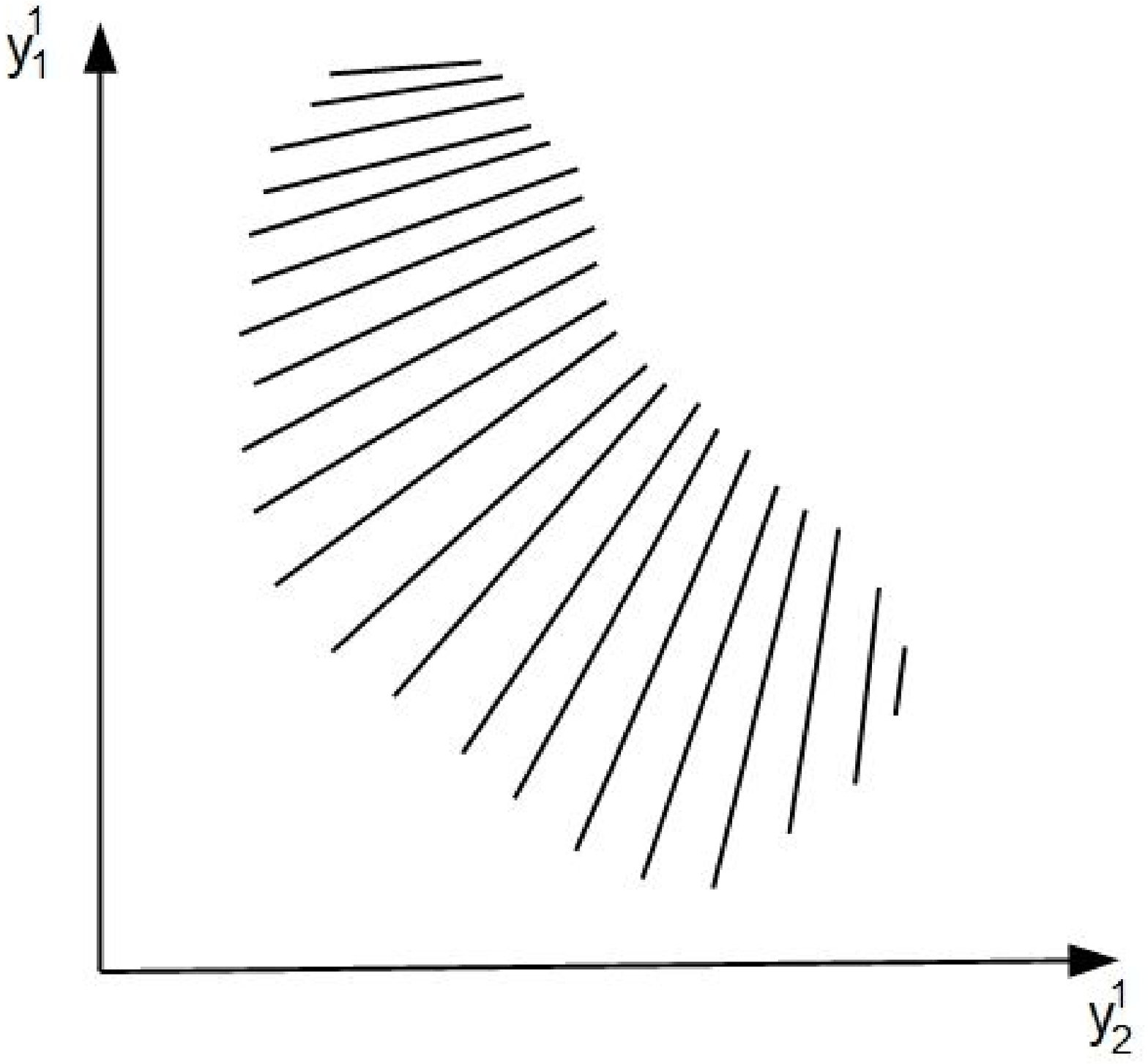}
\end{center}
\
\noindent and \(q_0(\lambda)\) determined by the foliation, determines the components
\(\eta^0_1,\,\eta^0_2\) and \(\xi\)
of a \LB\ field  exponentiable in \(\Jrr\).

As an example, take lines that intersect the  \(y^1_2\)-axis in a point \((0,\lambda)\) and whose inclination is \(\gamma\lambda\) parameterized by 
\(\lambda\), that is, \(m(\lambda) = \gamma\lambda\) and \( y^1_2=\gamma\lambda y^1_1+\lambda,\quad \lambda = {y^1_2}/(1+\gamma y^1_1)\)
\newcommand\Rat{\left(\frac{y^1_2}{1+\gamma y^1_1}\right)}
\newcommand\rat{\left(\frac{y^1_2}{y^1_1}\right)}

Thus, the equation \(f'_2(\lambda)-\gamma\lambda f'_1(\lambda)=\lambda g'(\lambda)\) becomes
\(f'_2(\lambda) = \gamma(q_0(\lambda )f_1(\lambda)D)'-\gamma f_1(\lambda) + g'(\lambda)\).
Calling \(F_1\) the primitive of \(f_1\),
we have \(f'_2=\gamma(\lambda F'_1)'-\gamma F'_1 + g'\),
that is, \(f_2=\gamma \lambda F'_1-\gamma F_1 +g\),
and thus, given arbitrary \(g\) and \(F_1\), we have that
\[\eta^0_1=F_1'(\lambda)= F'_1\Rat\]
\[\eta^0_2=\gamma\lambda F'_1\Rat-\gamma F_1\Rat + g\Rat\]
\[\xi =g\Rat\]
are the principal components of a \LB\ field
exponentiable in \(\Jrr\).

Another example is to take radial lines  parameterized by
 polar angle \(\theta\). We have then \(m(\theta) = \tan \theta\), that is,
\[\theta(y^1_1,y^1_2)=\arctan\rat.\]

Setting \(z = \tan \theta\), we have that the equation to be  solved is \(f'_2(z) = z f'_1(z)\).
Let \(F_1\) be the primitive of \(f_1\), therefore \(f_1=F'_1\); \(f_2=zF'_1-F_1\)
and thus, given arbitrary  \(F_1\) and \(g\)
\[\eta^0_1=F'_1\rat;\,\eta^0_2=\frac{y^1_2}{y^1_1} F'_1\rat -F_1\rat\]
and \(\ds \xi = g\rat\) where \(g\) is arbitrary,
are the principal
components of an exponentiable \LB\ field in \(\Jrr\) .
\hfill\eject

\section*{Acknowledgements} I once again express my highest consideration to George Svetlichny for his\linebreak extraordinary efforts in advising me on this dissertation. Along with consideration and respect, I express my sincere gratitude. This work received financial support from CAPES and CNPq, organs of the Brazilian government.



\begin{thebibliography}{xx}

\bibitem{olverLG} Peter J. Olver, \emph{Applications of Lie Groups to Differential Equations}, Springer, New York, 1993.

\bibitem{vino:smd14.661} A. M. Vinogradov, ``Multivalued solutions and a principIe of classification of non-linear differential .
equations", \emph{Soviet Math. Dokl.}, \textbf{14}, in 3 pp 661--665 (1973).

\bibitem{ibra-ande:SMD17.437} N. H. Ibragimov and R. L. Anderson, ``Groups of \LB\
contact transformations", \emph{Soviet Math. Dokl.}, \textbf{17}, no. 2, 437--441, (1976).

\bibitem{vino:SMD20.985} A. M. Vinogradov, ``The theory of higher infinitesimal
symetries of nonlinear partia1 differentia1 equations",
\emph{Soviet Math. Dokl.}, \textbf{20}, no. 5, 985--990, (1979).

\bibitem{otte-svet:JDE36.270} P. J. Otterson and G. Svetlichny, ``On derivative-dependent
infinitesimal deformations of differentiable maps",
\emph{J. Differential Equations}, \textbf{36}, no. 2, 270--294 (1980),

\bibitem{foka:JMAA68.347} A. S. Fokas, ``Group Theoretica1 Aspects of Constants of
Motion and Separab1e Solutions in Classica1 Mechanics".
\emph{J. Math. Anal. Appl.}, \textbf{68}, 347--370, (1979).


\bibitem{foca-lage:JMAA74.328} A. S. Fokas and P. A. Lagerstrom, ``Quadratic and Cubics
Invariants in Classical Mechanics", \emph{J. Math. Anal. Appl.}, \textbf{74}, 325--341 (1980).

\bibitem{foca-lage:JMAA74.342} A. S. Fokas and P. A. Lagerstrom, ``On the Use of \LB\
Operator' in Quantum Mechanics", \emph{J. Math. Anal. App1.},  \textbf{74}, 342--358 (1980).

\bibitem{ibra:CRASP293.657} N. H. Ibragimov, ``Sur l'\'equivalence des \'equations
d'evo1ution, qui admettent une a1g\`ebre de \LB\
infine",\emph{ C. R. Acad. Sc. Paris}, \textbf{293}, 657--660 (1981).


\bibitem{ovsj:SMD12.1497} L. V. Ovsjannikov, ``A nonlinear Cauchy problem in a scale of
Banach spaces",  \emph{Soviet Math. Dokl.} \textbf{12}, no. 5, 1497--1502, (1971).


\end{thebibliography}
\end{document}